\documentclass[a4paper, reqno, 11pt]{amsart}

\usepackage[english]{babel}
\usepackage{amsmath}
\usepackage{mathtools}
\usepackage{mathrsfs}
\usepackage{amssymb}
\usepackage{amsthm}
\usepackage{enumerate}
\usepackage{ifthen}
\usepackage{bbm}
\usepackage{color}
\usepackage{xcolor}
\usepackage{graphicx}
\provideboolean{shownotes} 
\setboolean{shownotes}{true}
\usepackage{hyperref}
\usepackage{comment}
\usepackage{geometry}
\usepackage[T1]{fontenc}
\usepackage{esint}

\newcommand{\margnote}[1]{
\ifthenelse{\boolean{shownotes}}%
{\marginpar{\raggedright\tiny\texttt{#1}}}%
{}%
}

\newcommand{\hole}[1]{
\ifthenelse{\boolean{shownotes}}%
{\begin{center} \fbox{ \rule {.25cm}{0cm}
\rule[-.1cm]{0cm}{.4cm} \parbox{.85\textwidth}{
\texttt{#1}} \rule {.25cm}{0cm}}\end{center}}
{}\bigskip
}
\newtheorem{mainthm}{Theorem}

\newtheorem{thm}{Theorem}[section]
\newtheorem{prop}[thm]{Proposition}
\newtheorem{lem}[thm]{Lemma}
\newtheorem{cor}[thm]{Corollary}
\newtheorem{defn}[thm]{Definition}

\theoremstyle{remark}
\newtheorem{rem}[thm]{Remark}

\newcommand{\e}{\varepsilon}		       
\newcommand{\R}{\mathbb{R}}

\newcommand{\N}{\mathbb{N}}
\newcommand{\Z}{\mathbb{Z}}

\newcommand{\dive}{\mathop{\mathrm {div}}}
\newcommand{\curl}{\mathop{\mathrm {curl}}}

\newcommand{\de}{\,\mathrm{d}}

\newcommand{\tr}{\,\mathrm{tr}}

\newcommand{\Qr}{Q_{\delta,N,r}}
\newcommand{\Qd}{Q_{\delta,N,\e}}

\newcommand{\inti}{\fint_{t_i}^{t_{i+1}}}

\numberwithin{equation}{section}

\usepackage{float}   
\usepackage{fancybox}		  
\usepackage{verbatim}

\usepackage{caption}
\usepackage{subcaption}

\newcommand{\schema}[1]{{\bf \sc #1}}
\newenvironment{pschema}[1]{\vspace{3mm}
\noindent\begin{Sbox}\begin{minipage}{.95\columnwidth}\vspace{2mm}\begin{center}{\large \schema{#1}}\vspace{5mm}\\
\begin{minipage}{0.9\textwidth}}{\end{minipage}\end{center}\vspace{2mm}\end{minipage}\end{Sbox}\fbox{\TheSbox}\vspace{3mm}}


\begin{document}

\title[$\theta$-method for the flow of nonsmooth velocity fields]{A priori error estimates for the $\theta$-method \\ for the flow of nonsmooth velocity fields}

\author[G. Ciampa]{Gennaro Ciampa}
\address[G.\ Ciampa]{DISIM - Dipartimento di Ingegneria e Scienze dell'Informazione e Matematica\\ Universit\`a  degli Studi dell'Aquila \\Via Vetoio \\ 67100 L'Aquila \\ Italy}
\email[]{\href{gciampa@}{gennaro.ciampa@univaq.it}}

\author[T. Cortopassi]{Tommaso Cortopassi}
\address[T.\ Cortopassi]{Scuola Normale Superiore\\ Piazza dei Cavalieri 7, 56126 Pisa\\ Italy.}
\email[]{\href{tcortopassi@}{tommaso.cortopassi@sns.it}}

\author[G. Crippa]{Gianluca Crippa}
\address[G. Crippa]{Departement Mathematik Und Informatik\\ Universit\"at Basel \\Spiegelgasse 1 \\CH-4051 Basel \\ Switzerland}
\email[]{\href{gianluca.crippa@}{gianluca.crippa@unibas.ch}}

\author[R. D'Ambrosio]{\\Raffaele D'Ambrosio}
\address[R. D'Ambrosio]{DISIM - Dipartimento di Ingegneria e Scienze dell'Informazione e Matematica\\ Universit\`a  degli Studi dell'Aquila \\Via Vetoio \\ 67100 L'Aquila \\ Italy}
\email[]{\href{raffaele.dambrosio@}{raffaele.dambrosio@univaq.it}}
\author[S. Spirito]{Stefano Spirito}
\address[S. Spirito]{DISIM - Dipartimento di Ingegneria e Scienze dell'Informazione e Matematica\\ Universit\`a  degli Studi dell'Aquila \\Via Vetoio \\ 67100 L'Aquila \\ Italy}
\email[]{\href{stefano.spirito@}{stefano.spirito@univaq.it}}

\begin{abstract}
Velocity fields with low regularity (below the Lipschitz threshold) naturally arise in many models from mathematical physics, such as the inhomogeneous incompressible Navier-Stokes equations, and play a fundamental role in the analysis of nonlinear PDEs.
The DiPerna-Lions theory ensures existence and uniqueness of the flow associated with a divergence-free velocity field with Sobolev regularity. In this paper, we establish a priori error estimates showing a logarithmic rate of convergence of numerical solutions, constructed via the $\theta$-method, towards the exact (analytic) flow of a velocity field with Sobolev regularity. In addition, we derive analogous a priori error estimates for Lagrangian solutions of the associated transport equation, exhibiting the same logarithmic rate of convergence. Our theoretical results are supported by numerical experiments, which confirm the predicted logarithmic behavior.
\end{abstract}

\maketitle

\section{Introduction}

In this paper we prove some new a priori error estimates for numerical solutions to the ordinary differential equation with nonsmooth velocity field and for the related linear transport equation. Additionally, we discuss the optimality of such error estimates by means of numerical simulations. As a numerical scheme, we consider the so-called $\theta$-method in a situation where the classical theory for Lipschitz velocity fields fails. More precisely, we consider divergence-free velocity fields with Sobolev regularity in the space variable, in which case a well-posedness theory for the ordinary differential equation and the transport equation is available thanks to the seminal contribution by DiPerna and Lions~\cite{DPL}, later extended by Ambrosio~\cite{Am} to the case of divergence-free velocity fields with bounded variation. 

\medskip

The need to go beyond Lipschitz regularity and to consider vector fields with Sobolev~\cite{DPL} or bounded variation~\cite{Am} regularity arises naturally in the study of nonlinear partial differential equations from mathematical physics. This is especially relevant in applications such as fluid dynamics~\cite{DPM,F, L1, L2}, kinetic theory~\cite{DPL2,DPL3}, and conservation laws~\cite{ABDL}. In the case when the velocity field is divergence-free, the transport equation coincides with the continuity equation, which, for example in the inhomogeneous incompressible Navier-Stokes equations, expresses the physical property of conservation of mass. The failure of Lipschitz regularity has both mathematical and physical motivations. On the mathematical side, energy estimates typically provide control of the velocity field only in an integral sense for its derivatives, rather than uniform bounds. On the physical side, turbulence theory predicts that fluid flows are often highly irregular, chaotic, and exhibit anomalous behaviors incompatible with a smooth regime.
\subsection{The flow of a velocity field and the classical $\theta$-method}
We begin by introducing some notation and recalling basic concepts. Let $b:[0,T]\times\R^{d}\to\R^{d}$ be a given velocity field, whose regularity assumptions will be specified later when needed. Given $T>0$, for any $x_0 \in \R^d$ we consider the Cauchy problem
\begin{equation}\label{eq:ode}
\tag{ODE}
\begin{cases}
\dot{X}(t,x_0)=b(t,X(t,x_0)),\\
X(0,x_0)=x_0 \,.
\end{cases}
\end{equation}
The unknown $X(\cdot,x_0):[0,T]\to\R^d$ is the trajectory of the velocity field $b$ starting at the point $x_0$ at the initial time~$t=0$. The map $X=X(t,x_0)$ as a function of both time and initial position is called the flow of the velocity field $b$. 

A widely used method to numerically compute solutions of~\eqref{eq:ode} is the so-called $\theta$-method, see \cite{dam}.\\
\\
\begin{pschema}{Classical $\theta$-method}
Given a parameter $\theta\in[0,1]$, for any $N\in\N$ define the time step $h\coloneqq T/N$ and consider the net~$\{t_i\}_{i=0}^{N}$ defined by $t_i\coloneqq i\,h$ for $i=0,\ldots,N$. The classical $\theta$-method is given by the following iteration.
\begin{itemize}
\item For any given $x_0 \in \R^d$, set $X_{0}^\theta(x_0)=x_0$. 
\item For $i=0,\ldots,N-1$, given $X_i ^\theta =X_i ^\theta (x_0) $ find $X^\theta_{i+1} =X_{i+1}^\theta (x_0)$ (in the following we omit explicitly writing the dependence on $x_0$ for the sake of  readability) by solving 
\begin{equation}\label{eq:thetamethod}
\frac{X_{i+1}^\theta-X_i^\theta}{h}=(1-\theta)  b(t_i,X_i^\theta )+ \theta  b(t_{i+1} ,X_{i+1}^\theta).
\end{equation}
\end{itemize}
\end{pschema}
\\
The $\theta$-method is a one-parameter family of numerical methods which reduces to the {\em explicit Euler method} for $\theta=0$ and to the {\em implicit Euler method} for~$\theta=1$. If the velocity field $b$ is bounded and globally Lipschitz continuous with Lipschitz constant~$L$, Gr\"onwall's inequality immediately implies the (global) error estimate
\begin{equation}\label{eq:apostlip}
\max_{0 \leq i \leq N}|X(t_i)-X^\theta_i|\lesssim h
\qquad \text{for any~$x_0 \in \R^d$},
\end{equation}
where the implicit constant depends, in particular, on the Lipschitz constant of $b$ and on $T$.
The estimate \eqref{eq:apostlip} can be improved with a $h^2$ at the right-hand side if $\theta=1/2$, since in such a case the method is known to be {\em unconditionally stable} and {\em second-order accurate}.
Thus, the order of convergence of the $\theta$-method may depend on the value of $\theta$, having that
\begin{itemize}
\item for $\theta\neq 1/2$, the method is first-order accurate;
\item for $\theta=1/2$ (i.e., the Crank-Nicolson method), it is second-order accurate.
\end{itemize}
This means that the global error behaves like $\mathcal{O}(h)$ in general, and $\mathcal{O}(h^2)$ only when $\theta=1/2$. However, if the vector field is Lipschitz continuous but not $C^2$, the order of convergence is $\mathcal{O}(h)$ for any $\theta\in[0,1]$.

\subsection{Error estimate for the flow of Sobolev velocity fields} 
Motivated by the applications to nonlinear problems discussed earlier, the main objective of this paper is to derive error estimates analogous to~\eqref{eq:apostlip} in the case of a velocity field that is bounded, divergence-free, and possesses Sobolev regularity in the spatial variable, namely, $b \in L^1 ((0,T) ; W^{1,p} (\R^d)) $ for a given $1<p < + \infty$. Since the velocity field is no longer defined pointwise everywhere (Sobolev functions are in general defined only almost everywhere), it is necessary to introduce a suitable weak notion of flow (the so-called regular Lagrangian flow, see Definition~\ref{def:rlf}). Moreover, in order to set up the iterative scheme in the $\theta$-method, we require an appropriate regularization of the velocity field (see Definition~\ref{def:approx}), in which the regularization parameter $\e$ will be coupled with the time step $h$. Finally, we introduce a suitable variant of the~$\theta$-method to handle the time dependence of the velocity field (see~\eqref{eq:theta4}). Within this framework, we are able to establish the following error estimate.\\
\\
\begin{pschema}{Error estimate for the flow}
Let $b \in L^1 ((0,T) ; W^{1,p} (\R^d))$ be a bounded divergence-free vector field for some $1<p<\infty$, and let $b_\e$ be a (suitable) smooth approximation of $b$ (for example the standard regularization by a smooth mollifier or more in general as in Definition \ref{def:approx}). Let $X$ be the unique regular Lagrangian flow of $b$ and $\{X_i^{\theta,\e}\}_i$ be the approximating flow constructed via the $\theta$-method as in \eqref{eq:thetamethod} with vector field $b_\e$. Then, there exists $\e=\e(h)$ such that on bounded domains one has
\begin{equation}\label{e:introflow}
\max_{0\leq i\leq N}\|X(t_i,\cdot)-X_i^{\theta,\e(h)}\|_{L^1}\lesssim \frac{1}{|\log h|},
\end{equation}
where $T/N=h$ and the implicit constant depends on $\|b\|_{L^1W^{1,p}}$.
\end{pschema}
\begin{rem}
In the case of a Lipschitz continuous vector field, one obtains a first-order convergence, as shown in \eqref{eq:apostlip}, for the 
%(pointwise) distance between the two trajectories. This pointwise estimate further implies a control on the uniform error, yielding a bound on the 
$L^\infty$ norm of the error.
In contrast, the error estimate in \eqref{e:introflow} yields a logarithmic rate of convergence in the case of a Sobolev vector field, and it refers to an averaged quantity, specifically the $L^1$ distance between the exact flow and its approximation.
\end{rem}

\medskip

The proof of the error estimate~\eqref{e:introflow} relies on an alternative approach to the theory by DiPerna-Lions~\cite{DPL}, developed in~\cite{CDL} by the third author in collaboration with C. De Lellis. While the DiPerna-Lions theory is based on the concept of renormalized solutions and on abstract compactness estimates, the theory in~\cite{CDL} is based on quantitative estimates on the regular Lagrangian flow, and as such, it provides effective rates of convergence. The logarithmic rate in~\eqref{e:introflow} corresponds to the logarithmic regularity established in~\cite{CDL}, which has been shown to be optimal in an analytical sense in~\cite{ACM} (see also~\cite{Seis2017} for a related approached based on optimal transportation theory and~\cite{BCC} for effective estimates in the zero-noise approximation of the regular Lagrangian flow). 

The rigorous statement of the error estimate~\eqref{e:introflow} will be given in Theorem~\ref{t:main1}  and after its statement we will make a few more technical comments on its assumptions and its significance. The proof of Theorem~\ref{t:main1} will be given in Section \ref{sect:teoA}. We will also complement this theoretical analysis by providing:
\begin{itemize}
\item some numerical simulations that will confirm the optimality of the rate~\eqref{e:introflow} for selected test problems, when the norm of the error is computed in a neighborhood of the singularity of the given vector field (Section~\ref{sect:simulations});
\item a numerical method to regularize the vector field (w.r.t. the space variable), which preserves the divergence-free condition and which is different from the standard regularization with smooth mollifiers (see Appendix~\ref{sec:discr-space}). This method is inspired by numerical schemes developed for the analysis of the two-dimensional Euler equations and it is constructed to ensure that the approximating field fulfills all the assumptions required by Theorem~\ref{t:main1}.
\end{itemize}

\subsection{The transport equation and error estimates for Lagrangian solutions} The classical theory of characteristics provides a link in the smooth setting between~\eqref{eq:ode} and the Cauchy problem for the linear transport equation, that is, the first-order partial differential equation
\begin{equation}\label{eq:te}
\tag{TE}
\begin{cases}
\partial_t u+b\cdot\nabla u=0,\\
u_{|_{t=0}}=u_0 \,.
\end{cases}
\end{equation}
More precisely, in the smooth setting~\eqref{eq:te} possesses a unique solution for every initial datum $u_0 : \R^d \to \R$, and such a solution has the form
\begin{equation}\label{def:sol-te}
u (t,x)\coloneqq u_0(X^{-1}(t,\cdot)(x)) \,.
\end{equation}
The solution is constant along trajectories of the velocity field, that is, it is a Lagrangian solution. 

Under Sobolev regularity assumptions on the velocity field, the DiPerna-Lions theory~\cite{DPL} guarantees that the connection between~\eqref{eq:ode} and~\eqref{eq:te} remains valid, provided we consider in~\eqref{def:sol-te} the (inverse of the) regular Lagrangian flow (as in Definition~\ref{def:rlf}). We note, however, that the Lagrangian property~\eqref{def:sol-te} for general weak solutions of~\eqref{eq:te} holds only under sufficiently large integrability assumptions on the weak solution. Specifically, if the velocity field $b \in L^1((0,T); W^{1,p}(\mathbb{R}^d))$, then~\eqref{def:sol-te} holds for solutions $u \in L^\infty ((0,T);L^q(\mathbb{R}^d))$ under the condition $1/p + 1/q \leq 1$. If this condition is not met, it is possible that non-unique solutions exist, and these solutions may not enjoy the Lagrangian property~\eqref{def:sol-te}. This issue will be discussed in more detail in Section \ref{sect:main}, along with some recent progress on this topic.

In Section \ref{sec:te}, we will prove the following a priori error estimate for the convergence towards the unique Lagrangian solution of~\eqref{eq:te} of a numerical approximation constructed with the generalized $\theta$-method.\\
\\
\begin{pschema}{Error estimate for the PDE}
Let $b \in L^1 ((0,T) ; W^{1,p} (\R^d))$ a bounded divergence-free vector field for some $1<p<\infty$ and $X$ the unique regular Lagrangian flow of $b$. Let $b_\e$ be a (suitable) smooth approximation of $b$, in the sense of Definition \ref{def:approx}, and let $u_0\in L^1(\R^d)$ be given. Then, there exist $\e=\e(h)$ and a numerical solution $u^h\in L^\infty((0,T);L^1(\R^d))$ (which depends on the approximating backward flow constructed via the $\theta$-method, see Definition \ref{def:sol-teN}) such that, on bounded domains 
\begin{equation}\label{e:introtransport}
\max_{0\leq i\leq N}\|u_0(X^{-1}(t_i,\cdot))-u^h(t_i,\cdot))\|_{L^1}\lesssim \frac{1}{|\log h|},
\end{equation}
where $T/N=h$ and the implicit constant depends on $\|b\|_{L^1W^{1,p}}$ and on the modulus of continuity in $L^1$ of $u_0$.
\end{pschema}
\begin{rem}
The error estimate for the flow in \eqref{eq:apostlip} implies that for Lipschitz vector field and Lipschitz initial datum, one obtains a first-order convergence for the $L^\infty$ norm of the error, namely
$$
\sup_{0\leq i\leq N}\|u(t_i,\cdot)-u^h(t_i,\cdot)\|_{L^\infty}\lesssim h,
$$
where the implicit constant depends on $\|b\|_{W^{1,\infty}}$ and $\mathrm{Lip}(u_0)$.
In contrast, the error estimate in \eqref{e:introtransport} yields a logarithmic rate of convergence in the case of a Sobolev vector field, and it refers to an averaged quantity, specifically the $L^1$ distance between the exact solution and its approximation.
\end{rem}

\subsection{Comparison with other numerical literature} 
The first works on the numerical approximation of the solution of the continuity equation with Sobolev velocity fields seem to be \cite{Walkington2005,Boyer2012}, where it has been shown the convergence of a discontinuous Galerkin method and a finite volume method respectively to the solution of \eqref{eq:te}, without any explicit rate of convergence. Later, a finite volume upwind method has been used in \cite{SeisSchlichting2017} and \cite{SeisSchlichting2018} to analyze the convergence of numerical schemes (an explicit and an implicit one, respectively) working directly on \eqref{eq:te}. These works provide explicit error estimates on a logarithmic Kantorovich-Rubinstein distance, obtaining a rate of order at least $1/2$ in the mesh sizes. Such results are based on a priori stability estimates proved in \cite{Seis2017}, which in turn are inspired by the a priori Lagrangian estimates in \cite{CDL}, and have recently been adapted to include diffusion in \cite{Navarro2022,Navarro2023}. Other works in this direction are \cite{BelgacemJabin2019,JabinZhou2024}, where the authors deal with nonlinear continuity equations or with non cartesian tessellation of the domain, a feature which for technical reasons was not previously achieved in \cite{SeisSchlichting2017}. 
Finally, we point out the work of the second author \cite{Cort}, where approximate solutions of \eqref{eq:ode} are constructed via an explicit Euler method 
%(which, as we will see, corresponds to the $\theta$-method when $\theta=0$) 
and correspondingly approximate solutions of \eqref{eq:te} that converge, w.r.t. a specific Wasserstein distance, to the unique distributional solution in the DiPerna-Lions setting with quantitative rates.
Moreover, the same logarithmic rate of convergence has been shown in \cite{Lucio}, for the Picard scheme applied to \eqref{eq:ode} with a divergence-free velocity field $b\in L^1((0,T);W^{1,p}(\R^d))$ with $p>d$. Finally, the convergence of finite difference schemes for stochastic transport equations has been investigated recently in \cite{FKP}.
%, under conditions that are less strict than those required for deterministic transport equations.

\subsection{Notation}
We fix here the notation we use in the paper. We denote by $L^p(\R^d)$ the standard Lebesgue spaces and with $\|\cdot\|_{L^p}$ their norm. %Moreover, $L^p_c(\R^d)$ denotes the space of $L^p$ functions defined on $\R^d$ with compact support. 
The Sobolev space of $L^p$ functions with distributional derivatives of first order in $L^p$ is denoted by $W^{1,p}(\R^d)$. 
The spaces $L^p_{\mathrm{loc}}(\R^d)$ and $W^{1,p}_{\mathrm{loc}}(\R^d)$ denote the space of functions which are locally in $L^p(\R^d)$ and $W^{1,p}(\R^d)$ respectively. 
Given a Banach space $\mathcal{X}$, we denote by $L^p((0,T);\mathcal{X})$ the Bochner space of all measurable functions $u:[0,T]\to \mathcal{X}$ endowed with its natural norm.
%$$
%\|u\|_{L^p((0,T);\mathcal{X})}\coloneqq \left(  \int_0^T\|u(t,\cdot)\|^p_{\mathcal{X}} \de t\right)^{\frac{1}{p}}<\infty,
%$$
%for all $1\leq p<\infty$ and and
%$$
%\|u\|_{L^\infty((0,T);\mathcal{X})}\coloneqq \esssup_{t\in [0,T]}\|u(t,\cdot)\|_{\mathcal{X}}<\infty,
%$$
%for $p=\infty$. 
The space of continuous functions on $[0,T]$ with values in $\mathcal{X}$ is denoted by $C([0,T];\mathcal{X})$, endowed with the supremum norm $\|\cdot\|_{L^\infty((0,T);\mathcal{X})}$. We denote by $B_r$ the ball of radius $r>0$ centered in the origin of $\R^d$ and with $\mathcal{L}^{d}$ the Lebesgue measure on $\R^d$. We also denote with $I$ the $d\times d$ identity matrix in $\R^{2d}$.
In the estimates we denote with $C$ a positive constant which may change from line to line. 
%We point out that if the velocity field $b\in W^{1,p}$ with $p>d$, we will always consider its H\"older continuous representative. 
Finally, given two positive quantities $A,B$, we sometimes use the notation $A \lesssim B$ whenever $A \leq CB$ with $C>0$ a uniform constant.

\section{Analytical tools, the numerical scheme, and main results}\label{sect:main}
The main goal of this paper is to provide a rigorous mathematical framework and to establish the a priori error estimate in~\eqref{e:introflow}, thus extending to the more general case of velocity fields with Sobolev regularity the classical error estimate~\eqref{eq:apostlip} which holds for Lipschitz velocity fields.
%can be derived using Gr\"onwall's inequality under the assumption of Lipschitz continuity of the velocity field. 
In this section, we present the precise mathematical setting for our analysis, review the relevant literature, state our main results, and provide some further comments. 

The starting point is that, in the nonsmooth setting, the classical notion of flow must be replaced by the concept of regular Lagrangian flow: 
\begin{defn}[Regular Lagrangian flow~\cite{DPL,Am}]\label{def:rlf}
Let $b\in L^1((0,T);L^1_{\mathrm{loc}}(\R^d))$ be a given velocity field. A function $X:(0,T)\times\R^d \to \R^{d}$ is a regular Lagrangian flow of $b$ if
\begin{itemize}
\item[$i)$]
for every  $\phi \in C_c ([0,T) \times \mathbb{R}^d)$ it holds that 
\begin{equation}\label{eq:form_integrale}
 \int_0 ^T \int_{\mathbb{R}^d} [X(t,x) \partial_t \phi (t,x) + b (t, X(t,x)) \phi(t,x)] \de t \de x + \int_{\mathbb{R}^d} x \phi(0,x) \de x =0.
 \end{equation}
\item[$ii)$] There exists a constant $L>0$ independent on $t$ such that 
\begin{equation}\label{eq:compress}
\mathcal{L}^{d}(X(t,\cdot)^{-1}(A))\leq L\mathcal{L}^{d}(A),
\qquad \text{for any open set $A\subset\R^{d}$.}
\end{equation}
\end{itemize}
\end{defn}

\begin{rem}[The role of condition~\eqref{eq:compress}]\label{rem_compressibility}
Condition~\eqref{eq:compress} in Definition~\ref{def:rlf} ensures that sets of positive Lebesgue measure remain of positive Lebesgue measure under the action of the flow, in a quantitative way. Notice that, for smooth divergence-free velocity fields, condition~\eqref{eq:compress}  is an equality and~$L=1$. Combining~\eqref{eq:form_integrale} and~\eqref{eq:compress}, it is possible to see that integral lines $t \mapsto X(t,x_0)$ of the velocity field~$b$ are well defined for~a.e.~$x_0$. More precisely, for a.e.~$x_0 \in \R^d$ it holds 
$$
X(t, x_0)- X(s,x_0) = \int_{s} ^{t} b (\tau, X(\tau, x_0)) \, \de \tau \quad \text{ for every $t,s >0$.}
$$
Therefore, condition~$i)$ in Definition~\ref{def:rlf} can be replaced by the requirement that $t~\mapsto~X(t, x_0)$ is an absolutely continuous solution of~\eqref{eq:ode} for a.e.~$x_0 \in \R^d$ (as in~\cite[Definition~6.1]{Am}).
\end{rem}

\begin{rem}[Condition~\eqref{eq:compress} as a selection condition]\label{rem_selection_RLF}
Condition $ii)$ in Definition~\ref{def:rlf} should be interpreted as a selection criterion. Indeed, the theory developed in~\cite{DPL,Am} establishes the existence of a unique flow for which both conditions $i),~ii)$ hold. We notice the following:
\begin{itemize}
\item If $p>d$ (in which case, in particular, by Sobolev embeddings, the velocity field admits a H\"older continuous representative, and is therefore well-defined pointwise everywhere), then by~\cite{CC} there exists a unique flow satisfying condition i), and condition ii) automatically holds as a consequence.
\item If $p \leq d$, however, there exist examples~\cite{BCDL, GS, Kumar, SP} of divergence-free velocity fields for which multiple flows satisfying only condition i) have been constructed (for a suitable, fixed pointwise representative of the velocity field). Among these flows, exactly one satisfies condition ii). In other words, while multiple a.e. flows may exist, there is a unique incompressible a.e. flow, that is, a unique regular Lagrangian flow. 
\end{itemize}
\end{rem}
The DiPerna-Lions theory~\cite{DPL} guarantees the existence of a unique regular Lagrangian flow in the case of a bounded, divergence-free velocity field $b \in L^1 ((0,T) ; W^{1,p} (\R^d)) $ for a given $1<p < + \infty$ (the boundedness of the velocity field and the condition on its divergence can be somehow relaxed, we refer to~\cite{DPL} for the details). This result has been extended by Ambrosio~\cite{Am} to bounded, divergence-free velocity field with bounded variation $b \in L^1 ((0,T) ; BV (\R^d))$. Given such a notion of regular Lagrangian flow, the first difficulty is how to suitably interpret the $\theta$-method iteration in~\eqref{eq:thetamethod}, since the velocity field is in general not defined pointwise everywhere (in the time variable, and in the space variable if $p \leq d$). Therefore, for the setup and the solvability of the $\theta$-method iteration, we need to regularize the velocity field, and in the error estimates the regularization parameter $\e$ will be coupled to the time step $h$. 
\begin{defn}[\textbf{Approximating sequence}]\label{def:approx}
Let $b\in L^1((0,T);W^{1,p}(\R^d)) $ for a given $1<p<\infty$ be a bounded, divergence-free velocity field. 
Given $\e>0$, we  say that a sequence of velocity fields~$\{b_{\e}\}_{\e}$  
%with $b_\e \in L^1 ((0,T);C^\infty (\R^d))$ {\color{red} what is this space?} 
is an {\em approximating sequence for $b$} if it satisfies
\begin{align}
&\|b_{\e}-b\|_{L^1_t L^1_x}\leq \bar{C}_0 \e^\alpha,\label{eq:r1}\\
&\|\nabla b_{\e}\|_{L^\infty_tL^\infty_x}\leq \frac{\bar C_1}{\e^\beta},\label{eq:r2}\\
&\|\dive b_{\e}\|_{L^\infty_tL^\infty_x}\leq \bar C_2,\label{eq:r3}\\
&\|b_\e\|_{L^\infty_tL^\infty_x}\leq \bar C_3,\label{eq:r4}
\end{align}
for some $\bar{C}_0, \bar C_1,\bar C_2, \bar C_3, \alpha,\beta>0$.
\end{defn}

\begin{rem} 
Our results can be generalized to the case in which~\eqref{eq:r1} and~\eqref{eq:r2} are replaced by
\begin{equation*}
\|b_\e-b\|_{L^1_tL^1_x}\leq \phi(\e)
\qquad \text{ and } \qquad
\|\nabla b_\e\|_{L^{\infty}_tL^\infty_x}\leq \frac{1}{\psi(\e)},
\end{equation*}
for  positive continuous functions $\phi,\psi:[0,1]\to\R^+$ with $\psi(0)=0$ and $\phi(0)=0$. In this case, the error estimate will depend on the function $\phi$ while the relation between $\e$ and $h$ will depend on the function $\psi$. 
\end{rem}

\begin{rem}
From a numerical point of view, the assumption on the divergence of $b_\e$ in \eqref{eq:r3} is quite restrictive if one wants to consider finite elements methods. Indeed, it is a well-known and important problem in numerical analysis to find finite elements bases with good control on the divergence, which is related to the non-local nature of the divergence-free constraint. Specifically, while the divergence of the (finite element) approximate field can be small in $L^2$, it may be not controlled in $L^\infty$. It is worth to note that the standard mollifying procedure by convolution satisfies \eqref{eq:r1}-\eqref{eq:r4}. In addition, in Appendix \ref{sec:discr-space} we provide another space discretization compatible with the assumptions in the main theorems based on the {\em vortex-blob method} for $2D$ Euler equations, see \cite{CCS3, DPM}.
\end{rem}
With the above definition, we define a regularized version of the $\theta$-method.\\
\\
\begin{pschema}{Regularized $\theta$-method}
For $x \in \R^d$, set~$X_{0}^{\theta,\e}(x)=x$, and for $i=0,\ldots,N-1$, given $X_{i}^{\theta,\e}= X_{i}^{\theta,\e}(x)$, compute $X_{i+1}^{\theta,\e}=X_{i+1}^{\theta,\e}(x)$ (as previously, we omit explicitly writing the dependence on $x$ where not necessary) by solving
\begin{equation}\label{eq:theta4}
\frac{X_{i+1}^{\theta,\e}-X_i^{\theta,\e}}{h}=(1-\theta)\fint_{t_i}^{t_{i+1}}b_{\e}(s,X_i^{\theta,\e})\de s+ \theta \fint_{t_i}^{t_{i+1}}b_{\e}(s,X_{i+1}^{\theta,\e})\de s.
\end{equation}
\end{pschema}
\\
We can now define $\tilde{X}^{h,\theta, \e}$ %through the same formula as in~\eqref{eq:theta3}
as
\begin{equation}\label{eq:theta-interpolato}
\tilde{X}^{h,\theta,\e}(t)\coloneqq \sum_{i=0}^{N-1}\left(X_{i}^{\theta,\e}+\frac{t-t_i}{h}(X_{i+1}^{\theta,\e}-X_i^{\theta,\e}) \right)\chi_{[t_i,t_{i+1})}(t).
\end{equation}

%Notice that $\tilde{X}^{h,\theta, \e}$ solves
%\begin{equation}\label{eq:contthetaregularised}
%\begin{cases}
%\partial_t \tilde{X}^{h,\theta,\e}(t)=\displaystyle\sum_{i=0}^{N-1}\left[\displaystyle\inti \left((1-\theta)b_\e(s,X_i^{\theta, \e})+\theta b_\e (s,X_{i+1}^{\theta, \e})\right)\de s\right]\chi_{[t_i,t_{i+1})}(t),\\
%\tilde{X}^{h,\theta,\e}(0)=x.
%\end{cases}
%\end{equation}

The rigorous statement of our a priori error estimate for the flow (recall~\eqref{e:introflow}) is the following: 
\begin{mainthm}[Error estimates for the flow]\label{t:main1} 
Let $b\in L^1((0,T);W^{1,p}(\R^d))\cap L^\infty((0,T)\times\R^d)$ be a divergence-free velocity field for some $1<p<\infty$. For $\e>0$, let $b_{\e}$ be an approximating sequence of the velocity field $b$ as in Definition~\ref{def:approx} and let $\e=h^\frac{1}{2\beta}$. Then, for any given $r>0$ there exists a positive constant $C=C(d,p,r,T,\|b\|_{L^\infty},\|\nabla b\|_{L^1_tL^p_x},\bar C_1,\bar C_2)$ such that, given $\tilde{X}^{h, \theta, \e}$ defined as in \eqref{eq:theta-interpolato}:
\begin{equation}\label{eq:rate2}
\int_{B_r}|\tilde{X}^{h,\theta,\e}(t,x)-X(t,x)|\de x\leq \frac{C}{|\log h|}.
\end{equation}
\end{mainthm}

\begin{rem}
We remark that the divergence-free assumption on the velocity field $b$ in Theorem~\ref{t:main1} can be relaxed to $(\dive b)^-\in L^1((0,T);L^{\infty}(\R^d))$ if $p>d$, and to $(\dive b)^{-}\in L^{\infty}((0,T)\times\R^d)$ if $1\leq p\leq d$, as done in \cite{Cort}. Moreover, Theorem~\ref{t:main1} holds true if $b\in L^1((0,T);W^{1,p}_{loc}(\R^d))$ and it satisfies suitable growth conditions (see \cite{DPL,CDL}).
\end{rem}

\begin{rem}[The case $p=1$]
If $p=1$, while the convergence of the scheme can be proved in a similar way as in Theorem~\ref{t:main1}, it is not clear how to obtain an explicit convergence rate. The reason is the lack of a strong $L^1$ estimate for the maximal function, an operator from harmonic analysis that plays a crucial role in the proof of Theorem~\ref{t:main1}. See however~\cite{BC} for estimates that depend on the equi-integrability in $L^1$ of the derivative of the velocity field, and the recent preprint~\cite{HS} which establishes a quantitative version of~\cite{Am} in terms of concatenated exponential rates.
\end{rem}

A second goal of this work is to apply the error estimates for the $\theta$-method to the Cauchy problem for the linear transport equation \eqref{eq:te}. By assuming that the initial datum $u_0 \in L^q(\R^d)$ for some $q \geq 1$, and the divergence-free velocity field $b \in L^1((0,T); W^{1,p}(\R^d))$ is bounded, with $1 < p < \infty$, in their seminal work, DiPerna and Lions \cite{DPL} showed that \eqref{eq:te} admits a unique distributional solution in $L^\infty((0,T);L^q(\R^d))$ whenever $1/p + 1/q \leq 1$. Thus, uniqueness also implies that the solution satisfies the representation formula \eqref{def:sol-te}. Notably, within the DiPerna-Lions framework, one can always define a unique Lagrangian solution, even without uniqueness of distributional solutions. Indeed, Modena, Sattig, and Székelyhidi \cite{MS3, MS} proved that if $1/p + 1/q \geq 1 + 1/d$, uniqueness fails for distributional solutions, thus breaking the link between \eqref{eq:ode} and \eqref{eq:te} in low integrability regimes. This non-uniqueness was recently extended in \cite{BCK} to the sharper range $1/p + 1/q \geq 1 + 1/(dq)$, which is optimal for {\em positive solutions} \cite{BCDL}. For sign-changing solutions, \cite{CCK} further improved this threshold to $1/p + 1/q > 1 + 1/d - \delta$, for some explicit $\delta = \delta(d,q) > 0$, leaving a gap between regimes. See also \cite{CL21, CL22} for non-uniqueness results in $L^r_t C^k_x$ spaces with $r < \infty$. 

%Nevertheless, Lagrangian solutions of \eqref{eq:te} remain uniquely defined and stable, even though there is no general selection principle via standard approximation techniques: this is the case of smooth approximations \cite{CCS2, DLG}, mollifications with smooth kernels and vanishing viscosity limits \cite{CCSo}.
The Lagrangian solution emerges as a natural candidate for a ``physically meaningful'' solution in this irregular setting, due to its uniqueness, stability, and renormalization properties. When only fractional regularity of the velocity field is assumed, several natural approximations (e.g., mollification, vanishing viscosity, or smooth approximations \cite{CCS2, DLG, CCSo}) can fail to select a unique distributional solution.

In our second theoretical result, we establish an a priori error estimate for the convergence of a numerical approximation towards the Lagrangian solution of the transport equation~\eqref{eq:te}. This numerical approximation is constructed by considering a backward Lagrangian flow, which can again be computed using the $\theta$-method. We emphasize that the resulting numerical solution $u^h$ does not, in general, solve the transport equation~\eqref{eq:te} exactly. Nevertheless, we show that $u^h$ converges to the unique Lagrangian solution $u^L$ of~\eqref{eq:te}.
\begin{mainthm}[Error estimates for the Lagrangian solution to the transport equation]\label{t:main2}
Let $b\in L^1((0,T);W^{1,p}(\R^d))\cap L^\infty((0,T)\times\R^d)$ be a divergence-free velocity field for some $1<p<\infty$ and let $Y$ be the backward regular Lagrangian flow solving \eqref{eq:back}. Define the sequence $\{Y_i^{\theta,\e}\}_{i,\e}$ as in \eqref{eq:theta back} and $\tilde Y^{h,\theta,\e}$ as in \eqref{eq:theta back 2}, corresponding to an approximating sequence $\bar b_t^\e$ of the {\em backward vector field} $\bar b_t$, as in \eqref{def:bar b}. Then, there exists a constant $c=C(d,p,T,\|b\|_{L^\infty_{t,x}},\|\nabla b\|_{L^1_tL^p_x},\bar C_1,\bar C_2)$ such that for any $u_0\in L^1(\R^d)$ the sequence $u^h$ defined in \eqref{def:sol-teN}, constructed via the backward flow $\tilde Y^{h,\theta,\e}$, converges to the unique Lagrangian solution $u^L$ of \eqref{eq:te}, defined in \eqref{def:sol-te}, with the estimate
\begin{equation}
\sup_{t\in[0,T]}\|u^{h}(t,\cdot)-u^L(t,\cdot)\|_{L^1}\leq C\left(\|u_0^\delta-u_0\|_{L^1}+\frac{C_\delta}{|\log h|}\right),
\end{equation}
where $u_0^\delta$ is any smooth approximation of $u_0$ with compact support and the constant $C_\delta$ blows up as $\delta\to 0$. 
\end{mainthm}

\section{Rate of convergence towards regular Lagrangian flows}
\label{sect:teoA}
The goal of this section is to prove Theorem \ref{t:main1}.
We first analyze the solvability of the regularized $\theta$-method~\eqref{eq:theta4}. In the case of the explicit Euler scheme, i.e. $\theta=0$, the maps $\{X_i^{0, \e}(x)\}_{i=0}^{N}$ are uniquely determined (by construction). So we just have to consider the case $0<\theta\leq 1$. We prove that if $\e$ and $h$ are chosen appropriately, the system of equations \eqref{eq:theta4} admits a unique smooth solution $\{X_i^{\theta,\e}(x)\}_{i=0}^{N}$, which in turn uniquely determines the map $\tilde{X}^{h, \theta, \e}$ defined in \eqref{eq:theta-interpolato}. 

\begin{prop}\label{prop:esistenza-p<d}
Let $b\in L^1((0,T);W^{1,p}(\R^d))\cap L^\infty((0,T)\times\R^d)$ be a divergence-free velocity field for some $1<p<\infty$ and let $T>0$, $N\in\N$ and $0\leq\theta\leq 1$ be given. Assume that $\{b_\e\}_\e$ is an approximating sequence as in Definition \ref{def:approx}.
Then, by considering $\e(h)=(2 \bar{C}_1h)^\frac{1}{\beta}$, with $\bar{C}_1$ and $\beta$ being the constants in \eqref{eq:r2}, there exists a sequence $\{X_i^{\theta,\e}(x)\}_{i=0}^{N}$ of continuous maps satisfying \eqref{eq:theta4}. 
\end{prop}
\begin{proof}
Let $\e>0$ and assume that the $i$-th element of the sequence $X_i^{\theta,\e}$ is given. We define the function
$$
f(z)\coloneqq z-\theta h \inti b_\e(s,z)\de s,
$$
and the vector 
$$
y=X_i^{\theta,\e}(x)+(1-\theta)h\inti b_\e(s,X_i^{\theta,\e}(x))\de s.
$$
Then, we define the element $X_{i+1}^{\theta,\e}$ of the sequence to be the solution of the system
$
f(z)=y.
$
Since $\nabla f(z)=I-\theta h \inti\nabla b_\e(s,z)\de s$, we can apply the inverse function theorem as long as
$$
h \|\nabla b_\e\|_{L^{\infty}}\leq \frac{1}{2},
$$
which holds thanks to our assumption on $\e$. Notice that if $\theta=0$ no condition on $\e$ is needed. 
\end{proof}

Now we show that we can control how the approximating flow obtained by solving the regularized $\theta$-method \eqref{eq:theta4} compresses the Lebesgue measure. The result is the following.

\begin{lem}
\label{lem:jac_num_irr}
Let $b\in L^1((0,T);W^{1,p}(\R^d))\cap L^\infty((0,T)\times\R^d)$ be a divergence-free velocity field with $1\leq p<\infty$. Assume that $\{b_\e\}_{\e>0}$ is an approximating sequence as in Definition \ref{def:approx}
% given sequence of smooth velocity fields satisfying \eqref{eq:r2}, \eqref{eq:r3} and \eqref{eq:r4} 
and let $\{X_i^{\theta,\e}(x)\}_{i,\e}$ be a sequence of continuous maps solving \eqref{eq:theta4}. Then, there exists a constant $C>0$ which depends on $\bar C_1, \bar C_2$ such that, if we set $\e(h)=h^\frac{1}{2\beta}$ with $h >0$ small enough (compared to C), the following estimate holds
\begin{equation}\label{stima jacobiano}
e^{-CT}\leq \det\nabla X_i^{\theta,\e}\leq e^{CT}, \qquad \mbox{for every }i=0,\ldots,N.
\end{equation}
In particular, for every $1\leq q\leq\infty$ and $f\in L^q(\R^d)$ we have that
\begin{equation}\label{eq:est1}
\|f(X_i^{\theta,\e})\|_{L^q(B_r)}\leq\|f\|_{L^q(B_R)}e^{\frac{CT}{q}},
\end{equation}
for every $i=0,\ldots,N$, $r>0$, and $R\coloneqq r+T\bar C_3$. 
\end{lem}
\begin{proof}
We first prove the estimate on the Jacobian of $X_i^{\theta,\e}$ \eqref{stima jacobiano} from below. We recall the classical formula for the determinant, see \cite[Chapter I, Section 3]{grothendieck1956theorie}
\begin{equation}\label{eq:det}
\det(I+h B)= 1+h\tr(A)+ \sum_{j=2} ^{d} h^j \sigma_j (\lambda_1 , \dots , \lambda_d),
\end{equation}
where $\sigma_j$ are symmetric polynomials of degree $j$ and $\lambda_i$ are the eigenvalues of $A$.
Then, from \eqref{eq:theta4} we know that
\begin{equation}
X_{i+1}^{\theta,\e}=X_i^{\theta,\e}+(1-\theta)h\inti b_\e(s,X_i^{\theta,\e})\de s+\theta h\inti b_\e(s,X_{i+1}^{\theta,\e})\de s,
\end{equation}
and computing the gradient it follows that
\begin{align*}
\nabla X_{i+1}^{\theta,\e}=\nabla X_i^{\theta,\e}&\left[I+(1-\theta) h\inti\nabla b_\e(s,X_i^{\theta,\e})\de s\right]\\
&+\theta h\nabla X_{i+1}^{\theta,\e}\inti\nabla b_\e(s,X_{i+1}^{\theta,\e})\de s.
\end{align*}
We move the last term above to the left hand side and computing the determinant we obtain the identity
\begin{align}
\det\nabla X_{i+1}^{\theta,\e} &\det\left[I-\theta h\inti\nabla b_\e(s,X_{i+1}^{\theta,\e})\de s\right]\nonumber\\
&=\det \nabla X_i^{\theta,\e} \det\left[I+(1-\theta) h\inti\nabla b_\e(s,X_i^{\theta,\e})\de s \right].\label{iterazione}
\end{align}
Furthermore, we use formula \eqref{eq:det} to compute
\begin{align}
\det\left[I-\theta h\inti\nabla b_\e(s,X_{i+1}^{\theta,\e})\de s\right]&=1-h \theta\inti \dive b_\e(s,X_{i+1}^{\theta,\e})\de s \nonumber\\
&+\sum_{j=2} ^{d} h^j \sigma_j (\lambda_1 , \dots, \lambda_d)= 1+ r_1 (\e,h),
\end{align}
and 
\begin{align}
\det\left[I+(1-\theta) h\inti\nabla b_\e(s,X_i^{\theta,\e})\de s\right]&=1+(1-\theta)h\inti\dive b_\e(s,X_i^{\theta,\e})\de s\nonumber\\
&+\sum_{j=2} ^{d} h^j \sigma_j (\lambda_1 , \dots, \lambda_d)= 1+ r_2 (\e,h).
\end{align}
We will now prove that, if $b_\e$ satisfies \eqref{eq:r2} and \eqref{eq:r3}, and provided that $\e$ is small enough, then there exists a uniform constant $\bar{C} >0$ such that
$$ 
| r_1 (\e, h)|, |r_2 (\e, h)| \leq \bar{C} h.
$$
Indeed, by \eqref{eq:r2} we have that $|\lambda_i| \lesssim \e^{- \beta} $ for every $i$, and thus $|\sigma_j(\lambda_1, \dots, \lambda_d)| \lesssim \e^{-j \beta}$. By the choice of $h(\e)=\e^{2 \beta}$ and \eqref{eq:r3} we thus have:
\begin{align}
   |r_1 (\e,h)| &= \left|-h \theta\inti \dive b_\e(s,X_{i+1}^{\theta,\e})\de s +\sum_{j=2} ^{d} h^j \sigma_j (\lambda_1 , \dots, \lambda_d) \right|\nonumber \\ 
   &\lesssim  \bar{C}_2 h +  \underbrace{h \e^{-2 \beta}}_{=1} h+ \underbrace{\sum_{j=3} ^{d} h^j \e^{-j \beta}}_{= o (h)} \leq \bar{C}h .   
\end{align}
The estimate on $r_2 (\e, h)$ follows in the same way. Thus, since $\det\nabla X_0^{\theta,\e}=1$, by iterating the chain of equalities \eqref{iterazione} and using the bounds on $r_1(\e, h)$ and $r_2 (\e, h)$ we just proved, we get that
\begin{equation}
\det\nabla X_{i+1}^{\theta,\e}\geq \frac{1}{(1+C h)^{i+1}}\geq \frac{1}{(1+C h)^{T/h}},
\end{equation}
provided that $h$ is small enough. Notice that the function
$
f(x)=(1+C_1 x)^{C_2/x}
$
is bounded and decreasing for $x\geq 0$, and moreover
$
\lim_{x\to 0}f(x)=e^{C_1 C_2}.
$
Thus, for $h$ small it holds $(1+Ch)^{\frac{T}{h}}\leq e^{CT}$, which eventually implies that
\begin{equation}
\det\nabla X_{i+1}^{\theta,\e}\geq e^{-CT}.
\end{equation}
With the very same argument we can obtain the upper bound on $\det\nabla X_{i+1}^{\theta,\e}$. The estimate in \eqref{eq:est1} follows from the change of variables formula for push-forward measures.
\end{proof}

The following lemma is the core of the proof of Theorem \ref{t:main1}.
\begin{lem}\label{lem:p>d}
Let $b\in L^1((0,T);W^{1,p}(\R^d))\cap L^\infty((0,T)\times\R^d)$ be a divergence-free velocity field, for some $1<p<\infty$, and let $X$ be its regular Lagrangian flow. Assume that $b_\e$ is an approximating sequence as in Definition \ref{def:approx}
%satisfying \eqref{eq:r1}, \eqref{eq:r2}, and \eqref{eq:r3} 
and let $\{X_i^{\theta,\e}(\cdot)\}_{i,\e}$ be defined as in \eqref{eq:theta4} and $\tilde X^{h,\theta,\e}$ the corresponding function defined as in \eqref{eq:theta-interpolato}. Moreover, for any given $r,\delta>0$ define the functional 
\begin{equation}
\Qr(t)\coloneqq \int_{B_r}\log\left(1+\frac{|\tilde{X}^{h,\theta,\e}(t,x)-X(t,x)|}{\delta}  \right)\de x.
\end{equation}
Then, there exists a constant $C>0$ which depends on $\bar C_1, \bar C_2$ such that, if we set $\e(h)=h^\frac{1}{2\beta}$ with $h >0$ small enough (compared to C), the following estimate holds
\begin{equation}\label{eq:stima-sopralivelli-p>d}
\sup_{t\in [0,T]}\Qr(t)\lesssim 
\frac{h^\frac{\alpha}{2\beta}}{\delta}+\left(1+\frac{h}{\delta}\right),
\end{equation}
where the implicit constant depends upon $d,p,r,T,\|b\|_\infty,\|\nabla b\|_{L^1_tL^p_x}$ and the constants appearing in Definition \ref{def:approx}.
\end{lem}

In the course of the proof of the lemma, we will need some tools from harmonic analysis. We first introduce the {\em Hardy-Littlewood maximal function}.
\begin{defn}\label{def:maxfunc}
Let $f\in L^1_{\mathrm{loc}}(\R^d)$, we define $Mf$ the maximal function of f as
$$
Mf(x)=\sup_{r>0}\frac{1}{\mathscr{L}^d(B_r)}\int_{B_r(x)}|f(y)|\de y, \hspace{1cm}\mbox{for every }x\in\R^d.
$$
\end{defn}
Note that the function $Mf$ is finite for a.e. $x\in\R^d$ if $f\in L^1(\R^d)$. 
%We also define the following quantity
%\begin{equation}
%|||f|||_{M^1}\coloneqq \sup_{\lambda>0}\{\lambda\mathscr{L}^d(\{x\in\R^d:|f(x)|>\lambda\}) \}.
%\end{equation}
The following estimates hold, for which we refer to \cite{Stein}.
\begin{lem}\label{est:maximal}
For every $1<p\leq\infty$ we have the strong estimate
$$
\|Mf\|_{L^p}\leq C_{d,p}\|f\|_{L^p}.
$$
%while just the following weak estimate holds for $p=1$
%$$
%|||Mf|||_{M^1}\leq C_d\|f\|_{L^1}.
%$$
\end{lem}

The maximal function operator is useful to estimate the difference quotients of a $W^{1,p}$ function.

\begin{lem}\label{lem:diff-quotients}
Let $f\in W^{1.p}(\R^d)$ then there exists a negligible set $\mathcal{N}\subset\R^d$ such that
\begin{equation}
|f(x)-f(y)|\leq C(d) |x-y|\left( M \nabla f(x)+M \nabla f(y) \right),
\end{equation}
for every $x,y\in\R^d\setminus \mathcal{N}$, where $\nabla f$ is the distributional derivative of $f$.
\end{lem}

\begin{proof}[Proof of Lemma~\ref{lem:p>d}]
To simplify the notations we define for $t\in [t_i,t_{i+1})$ the map
$$
\tilde X_i^{\theta,\e}(t,x)\coloneqq  X_{i}^{\theta,\e}(x)+\frac{t-t_i}{h}(X_{i+1}^{\theta,\e}(x)-X_i^{\theta,\e}(x)),
$$
so that $\tilde X^{h,\theta,\e}(t,x) =\tilde X_i^{\theta,\e}(t,x)$ if $t\in [t_i,t_{i+1})$. Moreover, we notice that $\tilde X^{h,\theta,\e}(t,x)=\tilde X_i^{\theta,\e}(t,x) = X_i ^{\theta , \e} (x)$ if $t= t_i$. We use the monotonicity and the concavity of the function $x\mapsto \log(1+x)$ and in particular that, for every $a,b\geq 0$, it holds
$$
\log(1+b)-\log(1+a)\leq [\log(1+x)']_{|x=a}|b-a|=\frac{|b-a|}{1+a},
$$
to compute
\begin{align*}
&\log\left(1+\frac{|\tilde{X}^{h,\theta,\e}(t,x)-X(t,x)|}{\delta}  \right)-\log\left(1+\frac{|\tilde{X}^{h,\theta,\e}(t_i,x)-X(t_i,x)|}{\delta} \right)\\
&\leq \frac{|\tilde{X}^{\theta,\e}_i(t,x)-X_i^{\theta,\e}(x)-(X(t,x)-X(t_i,x))|}{\delta+|\tilde X^{h,\theta,\e}(t_i,x) -X(t_i,x)|}\\
&=\frac{\left|(t-t_i)[(1-\theta)\hspace{-0.05cm}\fint_{t_i}^{t_{i+1}}\hspace{-0.05cm}b_\e(s,X_i^{\theta, \varepsilon} (x))\de s+ \theta \hspace{-0.05cm}\inti\hspace{-0.05cm} b_\e(s,X_{i+1}^{\theta, \varepsilon} (x))\de s]\hspace{-0.05cm}-\hspace{-0.05cm}\int_{t_i}^{t}\hspace{-0.05cm}b(s,X(s,x))\de s\right|}{\delta+|X_i^{\theta,\e}(x)-X(t_i,x)|}
\end{align*}
for every $t \in [t_i, t_{i+1})$. By iterating the estimate backward up to $t_0 = 0$ and using a telescopic argument, together with the bound $|t - t_i| \leq h$, we obtain the following estimate on each interval $(t_i, t_{i+1})$. Without loss of generality, we prove the claim for $t = T$. It holds:
\begin{align}
&\int_{B_r}\log\left(1+\frac{|\tilde{X}^{h,\theta,\e}(T,x)-X(T,x)|}{\delta}  \right)\de x\nonumber\\
&\leq \sum_{j=0}^{N-1}\int_{B_r}\frac{\left|\int_{t_j}^{t}[(1-\theta)b_\e(s,X_j^{\theta,\e}(x))+\theta b_\e(s,X^{\theta,\e}_{j+1}(x)) -b(s,X(s,x))]\de s\right|}{\delta+|X_j^{\theta,\e}(x) -X(t_j,x)|} \de x\nonumber\\
&\leq \sum_{j=0}^{N-1}\int_{B_r}\int_{t_j}^{t_{j+1}}\frac{|[(1-\theta)b_\e(s,X_j^{\theta,\e}(x))+\theta b_\e(s,X^{\theta,\e}_{j+1}(x)) -b(s,X(s,x))]|}{\delta+|X_j^{\theta,\e}(x)-X(t_j,x)|} \de x\de s.\label{cristo aiutami tu}
\end{align}
We split the inequality \eqref{cristo aiutami tu} as
\begin{align}
&\int_{B_r}\log\left(1+\frac{|\tilde{X}^{h,\theta,\e}(T,x)-X(T,x)|}{\delta}  \right)\de x\nonumber\\
&\leq \sum_{j=0}^{N-1}\int_{B_r}\frac{\int_{t_j}^{t_{j+1}}|[(1-\theta)b_\e(s,X_j^{\theta,\e}(x))+\theta b_\e(s,X^{\theta,\e}_{j+1}(x)) -b(s,X(s,x))]| \de s}{\delta+|\tilde{X}^{h,\theta,\e}(t_j,x)-X(t_j,x)|} \de x\nonumber\\
&\leq (1-\theta)\sum_{j=0}^{N-1}\int_{B_r}\int_{t_j}^{t_{j+1}}\frac{|b_\e(s,X_j^{\theta,\e}(x)) -b(s,X(s,x))|}{\delta+|X_j^{\theta,\e}(x)-X(t_j,x)|}\de s\de x \label{stima cacca 1}\\
&+\theta \sum_{j=0}^{N-1}\int_{B_r}\int_{t_j}^{t_{j+1}}\frac{|b_\e(s,X^{\theta,\e}_{j+1}(x)) -b(s,X(s,x))|}{\delta+|X_j^{\theta,\e}(x)-X(t_j,x)|} \de s\de x.\label{stima cacca 2}
\end{align}
%Since $p<d$ we cannot rely on the anisotropic estimate in Lemma \ref{lem:diff-quotients-anisotropico} and then we proceed as follows. 
We start by considering \eqref{stima cacca 1}: we add and subtract $b(s,X_j^{\theta,\e}(x))+b(s,X(t_j,x))$ and we get that
\begin{align*}
&\sum_{j=0}^{N-1}\int_{B_r}\int_{t_j}^{t_{j+1}}\frac{|b_\e(s,X_j^{\theta,\e}(x)) -b(s,X(s,x))|}{\delta+|X_j^{\theta,\e}(x)-X(t_j,x)|}\de s\de x\\
&\leq \sum_{j=0}^{N-1}\int_{t_j}^{t_{j+1}}\hspace{-0.2cm}\int_{B_r}\frac{|b_\e(s,X_j^{\theta,\e}(x))-b(s,X_j^{\theta,\e}(x))|}{\delta}\de x\de s\\
&+\sum_{j=0}^{N-1}\int_{t_j}^{t_{j+1}}\hspace{-0.2cm}\int_{B_r}\frac{|b(s,X_j^{\theta,\e}(x))-b(s,X(t_j,x))|}{|X_j^{\theta,\e}(x)-X(t_j,x)|}+\frac{|b(s,X(t_j,x))-b(s,X(s,x))|}{\delta}\de x\de s .
\end{align*}
Now we are going to repeatedly use Lemma \ref{lem:jac_num_irr} by changing variables, so the constants that appear will always depend on $e^{CT}$ and we will omit writing it explicitly.
For the first term on the right hand side above we have that
\begin{equation}\label{pezzo1}
\sum_{j=0}^{N-1}\int_{t_j}^{t_{j+1}}\int_{B_r}\frac{|b_\e(s,X_j^{\theta,\e}(x))-b(s,X_j^{\theta,\e}(x))|}{\delta}\de x\de s\lesssim \frac{\|b_\e-b\|_{L^1_tL^1_x}}{\delta}.
\end{equation}
For the second term, we use Lemma \ref{est:maximal}, H\"older's inequality and Lemma \ref{lem:diff-quotients} to obtain
\begin{align}\label{pezzo2}
\sum_{j=0}^{N-1}\int_{t_j}^{t_{j+1}}\hspace{-0.2cm}\int_{B_r}&\frac{|b(s,X_j^{\theta,\e}(x))-b(s,X(t_j,x))|}{|X_j^{\theta,\e}(x)-X(t_j,x)|}\de x\de s\lesssim \sum_{j=0} ^{N-1} \int_{t_j} ^{t_{j+1}} \hspace{-0.2cm}\int_{B_R} M|\nabla b| (s,x) \de x \de s \nonumber\\
& \leq \sum_{j=0} ^{N-1} \mathcal{L} ^d (B_R ) ^{\frac{p-1}{p}}  \int_{t_j} ^{t_{j+1}} ||M|\nabla b| (s, \cdot )||_{L^p (B_R)} \de x \de s  \lesssim  R^{\frac{p-1}{p}} ||\nabla  b||_{L^1 _t L^p _x} ,
\end{align}
where we recall that $R = r + T||b||_{L^ \infty} $. Lastly, for the third term we use again Lemma~\ref{est:maximal}, Lemma \ref{lem:diff-quotients}, the estimate
$$
|X(t_j,x))-X(s,x)|\leq \int_{t_j}^s|b(\tau,X(\tau,x))|\de\tau\leq h\|b\|_\infty, \qquad \mbox{for }s\in(t_j,t_{j+1}),
$$
and H\"older's inequality to get that 
\begin{equation}\label{pezzo3}
\sum_{j=0}^{N-1}\int_{t_j}^{t_{j+1}}\int_{B_r}\frac{|b(s,X(t_j,x))-b(s,X(s,x))|}{\delta }\de x\de s\lesssim R^\frac{p-1}{p}\frac{h}{\delta}\|b\|_{\infty}\|\nabla b\|_{L^1_tL^p_x}.
\end{equation}
The quantity in \eqref{stima cacca 2} can be handled with similar estimates adding and subtracting $b(s,X_{j+1}^{\theta,\e}(x))+b(s,X_j^{\theta,\e}(x))+b(s,X(t_j,x))$, which means we have to estimate one more piece. Precisely, we use the following estimate for each term in the sum:
\begin{align*}
&\frac{|b_\varepsilon (s,X_{j+1}^{\theta,\e}(x))-b(s,X(s,x))|}{\delta+|X_j^{\theta,\e}(x)-X(t_j,x)|} \leq \frac{|b_\varepsilon (s,X_{j+1}^{\theta,\e}(x)) -b(s,X_{j+1}^{\theta,\e}(x))|}{\delta+|X_j^{\theta,\e}(x)-X(t_j,x)|} \\
&\hspace{1cm} + \frac{|b(s,X_{j+1}^{\theta,\e}(x)) -b(s,X_j^{\theta,\e}(x)) |}{\delta+|X_j^{\theta,\e}(x)-X(t_j,x)|} +\frac{|b(s,X_j^{\theta,\e}(x)) -b(s,X(t_j,x))|}{\delta+|X_j^{\theta,\e}(x)-X(t_j,x)|}  \\
&\hspace{1cm}+  \frac{|b(s,X(t_j,x)) - b(s,X(s,x))|}{\delta+|X_j^{\theta,\e}(x)-X(t_j,x)|}. 
\end{align*}
The first, third and fourth terms in the above sum can be estimated as in \eqref{pezzo1}, \eqref{pezzo2} and \eqref{pezzo3} respectively. As for the third term, using again Lemma \ref{lem:diff-quotients}:
\begin{align*}
&\frac{|b(s,X_{j+1}^{\theta,\e}(x))-b(s,X_j^{\theta,\e}(x))|}{\delta+|X_j^{\theta,\e}(x)-X(t_j,x)|}\\
&\hspace{1cm}\lesssim \frac{|X_{j+1}^{\theta,\e}(x)-X_j^{\theta,\e}(x)|}{\delta}(M|\nabla b|(s,X_{j+1}^{\theta,\e}(x))+M|\nabla b|(s,X_{j}^{\theta,\e}(x)))\\
&\hspace{1cm} \lesssim\frac{h }{\delta} ||b||_{\infty}[M|\nabla b|(s,X_{j+1}^{\theta,\e}(x))+M|\nabla b|(s,X_{j}^{\theta,\e}(x))],
\end{align*}
where in the last line we used the definition \eqref{eq:theta4} of the regularized $\theta$-method and that $b \in L^ \infty ((0,T) \times \mathbb{R}^d)$. Then, integrating in space and time, using H\"older's inequality, Lemma \ref{est:maximal}, and Lemma \ref{lem:jac_num_irr} we get that 
\begin{equation}\label{pezzo4}
\sum_{j=0}^{N-1}\int_{t_j}^{t_{j+1}}\int_{B_r}\frac{|b(s,X_{j+1}^{\theta,\e}(x))-b(s,X_j^{\theta,\e}(x))|}{\delta+|X_j^{\theta,\e}(x)-X(t_j,x)|}\de x\de s\lesssim \frac{h }{\delta} ||b||_{\infty} R^\frac{p-1}{p}\|\nabla b\|_{L^1_t L^p_x}.
\end{equation}
Collecting the estimates in \eqref{pezzo1}, \eqref{pezzo2}, \eqref{pezzo3}, and \eqref{pezzo4} we finally obtain that 
\begin{align}
\sup_{t\in(0,T)}\Qd^r(t)&\lesssim \frac{\|b_\e-b\|_{L^1_tL^1_x}}{\delta}+  C(d,p,R)\left(1+\frac{h }{\delta} ||b||_\infty \right)\|\nabla b\|_{L^1_tL^p_x}\nonumber\\
& \lesssim \frac{h^\frac{\alpha}{2\beta}}{\delta}+ C(d,p,R) \left(\left(1+\frac{h}{\delta}\right)\|\nabla b\|_{L^1_tL^p_x} \right)\label{rem:lp}
\end{align}
where we used the assumption \eqref{eq:r1} and the regularization of $b$ and the definition of $h$. The proof is complete.
\end{proof}

\begin{rem}
Notice that one can estimate the term \eqref{pezzo1} as follows
\begin{equation}
\sum_{j=0}^{i}\int_{t_j}^{t_{j+1}}\int_{B_r}\frac{|b_\e(s,X_j^{\theta,\e}(x))-b(s,X_j^{\theta,\e}(x))|}{\delta}\de x\de s\leq CR^{\frac{p-1}{p}}\frac{\|b_\e-b\|_{L^1_tL^p_x}}{\delta},
\end{equation}
with $R\coloneqq r+T\|b\|_\infty$. Thus, Lemma \ref{lem:p>d}, and consequently also Theorem \ref{t:main1}, hold if we consider the condition
\begin{equation}
    \|b_\e-b\|_{L^1_tL^p_x}\leq C\e^\alpha
\end{equation}
for some $C>0$, bounding the $L^1 _t L^p _x$ norm instead of the $L^1 _t L^1 _x$ norm as in \eqref{eq:r1}. 
\end{rem}

\begin{rem}
The lemma just proven could be simplified in the case $p>d$. In this case, one could use an anisotropic version of Lemma \ref{lem:diff-quotients} (see \cite[Lemma 5.1]{CC}), as done for the explicit Euler method \cite{Cort}, which would avoid the need for a change of variable in the approximated flow. However, in order to define the maps $X^{\theta, \e}_i$ via Proposition \ref{prop:esistenza-p<d}, the spatial approximation is necessary, which implies a restriction on $h$ and $\e$. Therefore, the use of the anisotropic estimate does not provide any real advantage, since spatial approximation remains a fundamental requirement.
\end{rem}

We can now prove Theorem \ref{t:main1}.
\begin{proof}[Proof of Theorem \ref{t:main1}]
First of all, let $\eta>0$ and define the set 
$$
A_\eta\coloneqq \{x\in \R^d:|\tilde{X}^{h,\theta,\e}(t,x)-X(t,x)|\geq \eta \}.
$$
Then, for any fixed $r>0$, the following estimate holds
\begin{align*}
\int_{B_r}|\tilde{X}^{h,\theta,\e}(t,x)-X(t,x)|\de x &\leq C\eta r^d+\|\tilde X^{h,\theta,\e}(t,\cdot)\|_{L^\infty(B_r)}\mathscr{L}^d(B_r\cap A_\eta)\\
&+\|X(t,\cdot)\|_{L^\infty(B_r)}\mathscr{L}^d(B_r\cap A_\eta)
\end{align*}
We apply Lemma \ref{lem:p>d} in order to bound the second term above: we obtain that
\begin{align}
\mathscr{L}^d\{ x\in B_r: |\tilde{X}^{h,\theta,\e}(t,x)-X(t,x)|>\eta \}&\leq \frac{\Qr(t)}{\log(1+\eta/\delta)}\nonumber\\
& \leq \frac{C}{\log(1+\eta/\delta)}\left(\frac{\e^\alpha}{\delta}+1+\frac{h(\e)}{\delta} \right).\label{stima sopralivelli}
%&\leq C(d,p,r,T,\bar C_2,\bar C_3)\left(\frac{ \|\nabla b\|_{L^1L^p}}{\log(1+\eta/\delta)}+\frac{\|b\|_{L^\infty}}{\log(1+\eta/\delta)}\frac{1}{N\delta}\right),
\end{align}
We define $\delta(\e)\coloneqq \max\left\{\e^\alpha,h(\e)\right\}$. Thus, it follows that
\begin{equation}
\int_{B_r}|\tilde{X}^{h,\theta,\e}(t,x)-X(t,x)|\de x\leq C\eta r^d+\frac{C(d,p,r,T,\|b\|_{L^\infty},\|\nabla b\|_{L^1L^p})}{\log(1+\frac{\eta}{\delta(\e)})}.
\end{equation}
Finally, by choosing $\eta=\sqrt{\delta(\e)}$ we obtain
\begin{equation}\label{est:stimaL1-p>d}
\int_{B_r}|\tilde{X}^{h,\theta,\e}(t,x)-X(t,x)|\de x\leq C r^d\sqrt{\delta(\e)}+\frac{C(d,p,r,T,\|b\|_{L^\infty},\|\nabla b\|_{L^1L^p})}{|\log \delta(\e)|},
\end{equation}
and then by using the (polynomial) relation between $\e$ and $h$, the rate of convergence is given by the slowest term, namely $|\log h|^{-1}$. This concludes the proof. 
\end{proof}

\begin{rem}
We require the bounds \eqref{eq:r2} and \eqref{eq:r3} for the approximating sequence $b_\e$, rather than the more natural $L^1_tL^\infty_x$ condition, in order to obtain an explicit convergence rate in terms of $h$. The $L^1_tL^\infty_x$ bound would still ensure convergence, but the constants would depend on the equi-integrability in time of $\|\nabla b_\e(t,\cdot)\|_\infty$ and $\|\dive b_\e(t,\cdot)\|_\infty$, preventing explicit rates
\end{rem}

\section{Numerical simulations}\label{sect:simulations}
Let us complement the presented rigorous analysis with selected numerical experiments related to the error estimate provided in Theorem~\ref{t:main1}. To this purpose, we consider problem \eqref{eq:ode} with the velocity field
\begin{equation}\label{eq:prtest}
b(x_1,x_2)=\left(
-2(\alpha+1)(x_1^2+x_2^2)^{\frac{\alpha-1}{2}}x_2,\,
2(\alpha+1)(x_1^2+x_2^2)^{\frac{\alpha-1}{2}}x_1
\right)\in W^{1,p}_\mathrm{loc}(\R^2),
\end{equation}
where $1-\frac2p<\alpha<1$ and $p>2$. The vector field in \eqref{eq:prtest} is smooth away from the origin, where it has a singularity, so no spatial regularization is needed. The associated ODE system admits the explicit solution
\begin{equation}
\label{soluzione esplicita}
\begin{cases}
x_1(t)=r_0\cos\left(\beta_0+2(\alpha+1)r_0^{\alpha-1}t\right)\\
x_2(t)=r_0\sin\left(\beta_0+2(\alpha+1)r_0^{\alpha-1}t\right)
\end{cases}
\end{equation}
where $r_0=\sqrt{x_1(0)^2+x_2(0)^2}$ and $\beta_0=\arctan\left(\frac{x_2(0)}{x_1(0)}\right)$.
We apply the $\theta$-method \eqref{eq:theta4} for selected values of $h$ and $\theta\in[0,1]$. Depending on the choice of the initial condition, we observe the following behavior:
\begin{itemize}
\item When the initial condition is placed sufficiently far from the origin, the convergence rate appears to be faster than $1/|\log h|$, yet still slower than linear. Moreover, the trajectories obtained by the $\theta$-method are circles, in agreement with the exact solution of the system.
\item When the initial datum is chosen in a small neighborhood of the origin, the observed convergence rate aligns with the logarithmic rate predicted by our theoretical analysis.
\end{itemize}
Figure \ref{spettacolo3} displays the pattern of the error estimate and highlight its decay in comparison with the reference logarithmic slope.\\

\begin{figure}[htbp]
    \centering
    \begin{subfigure}{0.45\textwidth}
        \includegraphics[width=\linewidth]{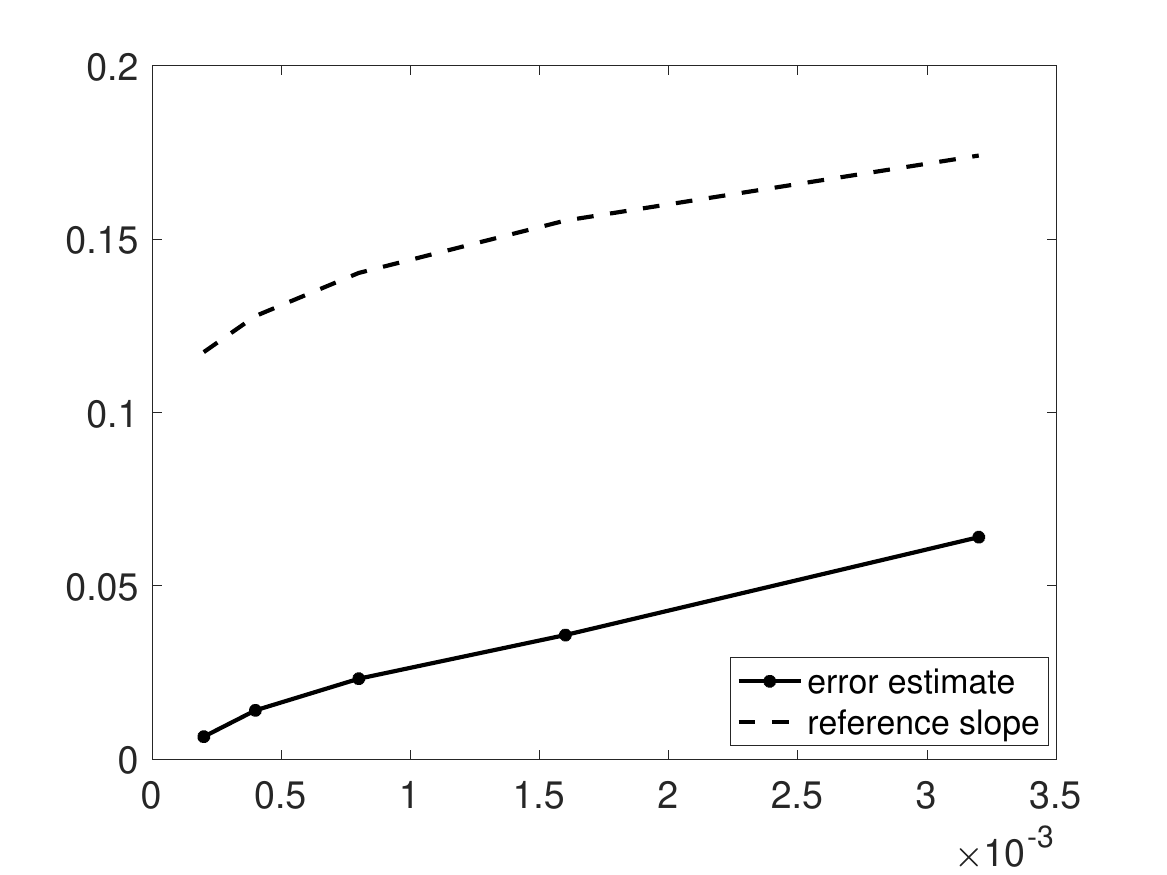}
        \caption{$\theta=0.2$, $p=3$, $\alpha=0.36$.}
        \label{spettacolo}
    \end{subfigure}
    \hfill
    \begin{subfigure}{0.45\textwidth}
        \includegraphics[width=\linewidth]{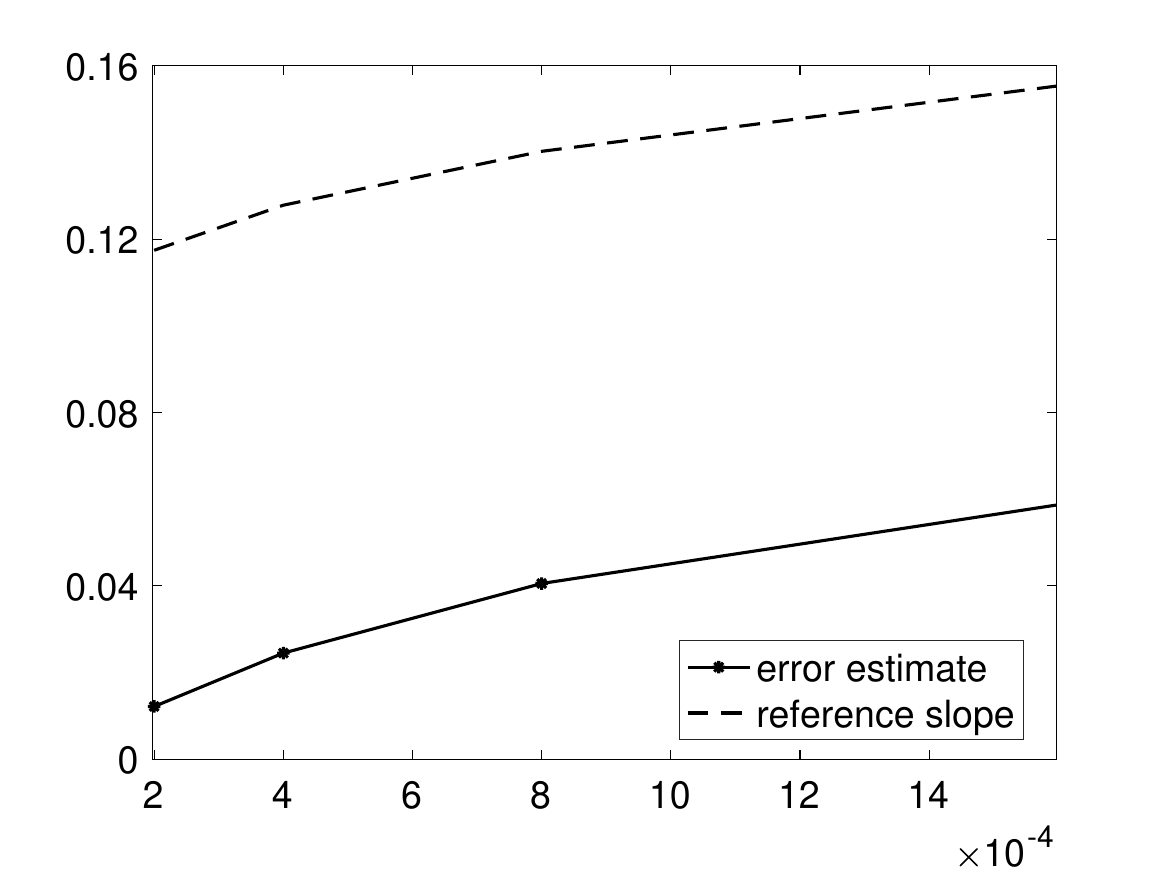}
        \caption{$\theta=0$, $p=3$, $\alpha=0.35$.}
        \label{spettacolo2}
    \end{subfigure}
    \caption{Error estimate computed according to \eqref{eq:rate2}, arising from the application of the $\theta$-method to Equation \eqref{eq:prtest}. The reference solution is \eqref{soluzione esplicita} and the displayed error decay results from the application of the $\theta$-method with step-sizes $2^kh_0$, $k=0,1,2,3$, with $h_0=10^{-4}$.}
    \label{spettacolo3}
\end{figure}

%\begin{figure}[htbp]
%\begin{center}
%\scalebox{0.25}{\includegraphics{prob1a_th02_alf035_h1em4_y0001}}
%\caption{Error estimate computed according to \eqref{eq:rate2}, arising from the application of the $\theta$-method to Equation \eqref{eq:prtest}, for $\theta=0.2$, assuming $\alpha=0.35$. The reference solution is the explicit solution and the displayed error decay results from the application of the same method with step-sizes $2^kh_0$, $k=0,1,2,3$.}
%\label{spettacolo}
%\end{center}
%\end{figure}
%
%\begin{figure}[htbp]
%\begin{center}
%\scalebox{0.25}{\includegraphics{prob1a_th0_alf034_h1em4_y0001}}
%\caption{Error estimate computed according to \eqref{eq:rate2}, arising from the application of the $\theta$-method to Equation \eqref{eq:prtest}, for $\theta=0$, assuming $\alpha=0.35$. The reference solution is the explicit solution and the displayed error decay results from the application of the same method with step-sizes $2^kh_0$, $k=0,1,2,3$.}
%\label{spettacolo2}
%\end{center}
%\end{figure}

In the plots, we compute the difference between the trajectory (with step size $h$) and the reference trajectory given by \eqref{soluzione esplicita}.
This difference is evaluated at fixed initial data, and should be interpreted as a pointwise error between trajectories, rather than a norm-based error (such as the $L^1$-norm).
Nevertheless, this pointwise comparison is equivalent to evaluating the error in the $ L^\infty_t L^1_x $ norm, because when implementing a quadrature rule to approximate the spatial integral (i.e., the $L^1$ norm) we evaluate the difference between the two solutions at mesh points. Thus, the pointwise distance between trajectories at discrete times naturally corresponds to a discrete approximation of the desired norm.
%\begin{figure}[htbp]
%\begin{center}
%\scalebox{0.25}{\includegraphics{prob1a_th03_alf0333_h1em4}}
%\caption{Error estimate computed according to \eqref{eq:rate2}, arising from the application of the $\theta$-method to Equation \eqref{eq:prtest}, for $\theta=0.3$, assuming $\alpha=0.34$. The reference solution is computed with step-size $h_0=10^{-4}$ and the displayed error decay results from the application of the same method with step-sizes $2^kh_0$, $k=0,1,2,3$.}
%\label{fig:prtest1}
%\end{center}
%\end{figure}
Concerning the numerical trajectories, they exhibit a spiraling behavior: they move away from the origin when $\theta < 1/2$, and converge towards the origin when $\theta > 1/2$. This qualitative behavior is not consistent with the exact solution of the system. On the other hand, for $\theta=\frac12$ the trajectories are circles. 
\\
%Additionally, we observe that the distance between the numerical and exact solutions grows linearly with respect to $T$, see Figure \ref{fig:pr}.
%\begin{figure}[htbp]
%\begin{center}
%\scalebox{0.07}{\includegraphics{temp_fig1.png}}
%\caption{Trajectory corresponding to the problem analyzed in Figure \ref{fig:prtest1} with initial datum $(0.01,0.01)$.}
%\label{fig:pr}
%\end{center}
%\end{figure}
%\begin{figure}[htbp]
%\begin{center}
%\scalebox{0.3}{\includegraphics{prob1a_th0_alf08_h1em4}}
%\caption{Error estimate computed according to \eqref{eq:rate2}, arising from the application of the $\theta$-method to Equation \eqref{eq:prtest}, for $\theta=0$, assuming $\alpha=0.8$. The reference solution is computed with step-size $h_0=10^{-4}$ and the displayed error decay results from the application of the same method with step-sizes $2^kh_0$, $k=0,1,2,3$.}
%\label{fig:prtest2}
%\end{center}
%\end{figure}
%
%\begin{figure}[htbp]
%\begin{center}
%\scalebox{0.3}{\includegraphics{prob1a_th08_alf08_h1em4}}
%\caption{Error estimate computed according to \eqref{eq:rate2}, arising from the application of the $\theta$-method to Equation \eqref{eq:prtest}, for $\theta=0.8$, assuming $\alpha=0.8$. The reference solution is computed with step-size $h_0=10^{-4}$ and the displayed error decay results from the application of the same method with step-sizes $2^kh_0$, $k=0,1,2,3$.}
%\label{fig:prtest3}
%\end{center}
%\end{figure}
\\
We now consider problem \eqref{eq:ode} with velocity fields
\begin{equation}\label{eq:prtest-4}
b(x_1,x_2)=\left(
\sqrt{|\sin 2\pi x_2|},\,
\sqrt{|\sin 2\pi x_1|}\right)\in W^{1,q}_\mathrm{loc}(\R^2),\,\,\mbox{with } 1\leq q<2,
\end{equation}
and
\begin{equation}\label{eq:prtest-5}
b(x_1,x_2)=\left(
\log\left(\frac{1}{|x_2|}\right)^\frac1p |x_2|^{1-\frac{1}{2p}},
\log\left(\frac{1}{|x_1|}\right)^\frac1p |x_1|^{1-\frac{1}{2p}}\right)\in W^{1,p}_\mathrm{loc}(\R^2),
\end{equation}
for some $p>1$ to be chosen, and initial datum in a small neighborhood of the origin. Figure \ref{altro} displays the pattern of the error estimate and highlight its decay in comparison with the reference logarithmic slope. Our numerical procedure is as follows. 
\begin{enumerate}
\item We fix a time step $h$ and solve the system using the $\theta$-method.
\item We vary the step size and treat the numerical solution corresponding to the smallest step size, denoted by $h_0$, as our reference (exact) solution.
\end{enumerate}
In the plots, we compute the difference between the trajectory obtained with step size $h$ and the reference trajectory corresponding to $h_0$, confirming the feasibility of the error estimate in \eqref{eq:rate2}.

\begin{figure}[htbp]
    \centering
    \begin{subfigure}{0.45\textwidth}
        \includegraphics[width=\linewidth]{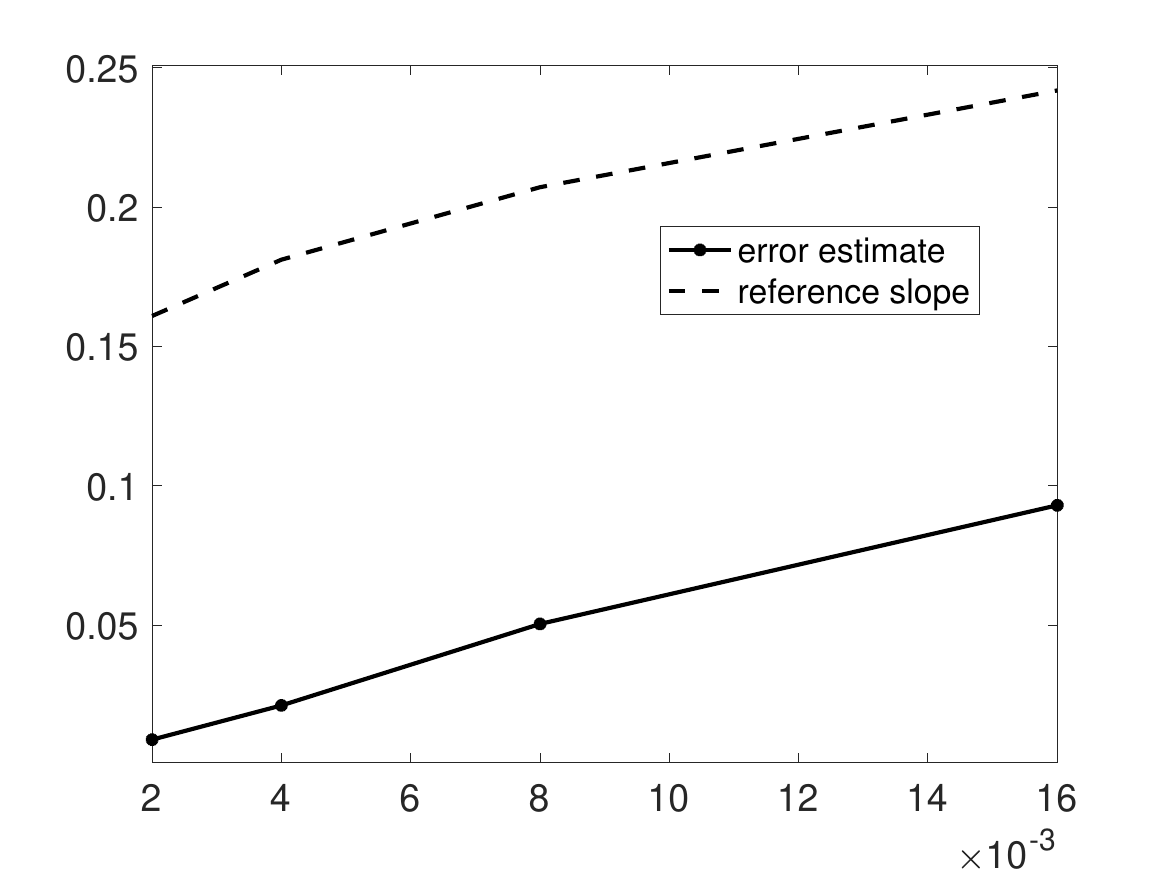}
        \caption{$\theta=0.8$, $h_0=10^{-3}$.}
        \label{altro1}
    \end{subfigure}
    \hfill
    \begin{subfigure}{0.45\textwidth}
        \includegraphics[width=\linewidth]{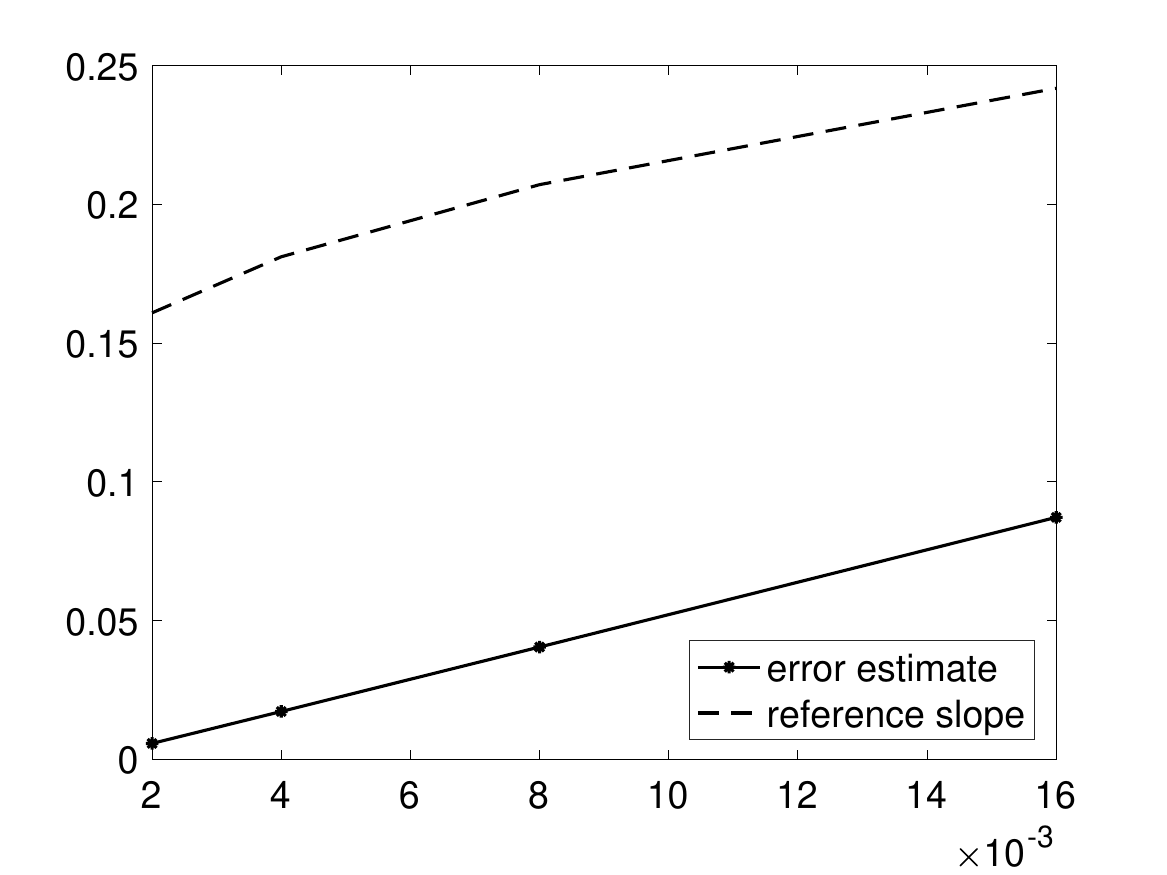}
        \caption{$\theta=0.7$, $p=100$, $h_0= 10^{-3}$.}
        %\includegraphics[width=\linewidth]{prob5_th04_p3_h2em3}
        %\caption{$\theta=0.4$, $p=3$, $h_0=2\cdot 10^{-3}$.}
        \label{altro2}
    \end{subfigure}
    \caption{Error estimate computed according to \eqref{eq:rate2}, arising from the application of the $\theta$-method to Equation \eqref{eq:prtest-4} (on the left) and Equation \eqref{eq:prtest-5} (on the right). The reference solution is computed with step-size $h_0$ and the displayed error decay results from the application of the same method with step-sizes $2^kh_0$, $k=0,1,2,3$.}
    \label{altro}
\end{figure}

%\begin{figure}[htbp]
%\begin{center}
%\scalebox{0.25}{\includegraphics{prob4_th03_p101_h2em4}}
%\caption{Error estimate computed according to \eqref{eq:rate2}, arising from the application of the $\theta$-method to Equation \eqref{eq:prtest-4}, for $\theta=0.3$, assuming $p=1.01$. The reference solution is computed with step-size $h_0=2\cdot 10^{-4}$ and the displayed error decay results from the application of the same method with step-sizes $2^kh_0$, $k=0,1,2,3$.}
%\label{fig:prtest4}
%\end{center}
%\end{figure}
%\begin{figure}[htbp]
%\begin{center}
%\scalebox{0.3}{\includegraphics{prob5_th04_p3_h2em3}}
%\caption{Error estimate computed according to \eqref{eq:rate2}, arising from the application of the $\theta$-method to Equation \eqref{eq:prtest-5}, for $\theta=0.4$, assuming $p=3$. The reference solution is computed with step-size $h_0=2\cdot 10^{-3}$ and the displayed error decay results from the application of the same method with step-sizes $2^kh_0$, $k=0,1,2,3$.}
%\label{fig:prtest5}
%\end{center}
%\end{figure}

\section{Lagrangian solutions of the transport equation} \label{sec:te}
In this section we apply the Lagrangian theory developed in the previous sections to show that we can use the $\theta$-method to construct approximating sequences converging to the unique Lagrangian solution of \eqref{eq:te}.
In order to deal with the transport equation, we need to introduce the backward approximating flow.  
For every $t\in [0,T]$ we consider the regular Lagrangian flow $Y(t,\cdot,\cdot)$ of
\begin{equation}\label{eq:back}
\begin{cases}
\partial_s{Y}(t,s,x)=-b(t-s,{Y}(t,s,x)),\\
Y(t,0,x)=x,
\end{cases}
\end{equation}
so that, if $X$ is the regular Lagrangian flow of $b$, we have the relation $X^{-1}(t,\cdot)=Y(t,t,\cdot)$. With these notations, \eqref{eq:back} is a forward ODE which keeps track of the final time $t$. Then, we can write the $\theta$-method for \eqref{eq:back} as follows. Let $t\in [0,T]$, $\theta\in[0,1]$, and for any fixed $N\in\N$ define the time step $h\coloneqq T/N$ and consider the net $\{t_i\}_{i=0}^{N}$ defined by $t_i\coloneqq i\,h$, $i=0,\ldots,N$. For simplicity, we also define 
\begin{equation}\label{def:bar b}
\bar b_t(s,x)=\begin{cases}
-b(t-s,x),\qquad &\mbox{ if }s\leq t\\
0, &\mbox{ otherwise. }
\end{cases}
\end{equation}
Then, we proceed iteratively as follows.
Set $Y_{0}^{\theta,\e}(t,x)=x$. For $i\in \{0,\ldots,N-1\}$, given $Y_{i}^{\theta,\e}=Y_{i}^{\theta,\e}(t,x)$, we find $Y_{i+1}^{\theta,\e}=Y_{i+1}^{\theta,\e}(t,x)$ by solving 
\begin{equation}\label{eq:theta back}
\frac{Y_{i+1}^{\theta,\e}-Y_i^{\theta,\e}}{h}=(1-\theta) \fint_{t_i}^{t_{i+1}} \bar b_t^\e(\tau,Y_i^{\theta,\e})\de \tau+ \theta \fint_{t_i}^{t_{i+1}} \bar b_t^\e(\tau,Y_{i+1}^{\theta,\e})\de \tau,
\end{equation}
where we omit writing explicitly the dependence on $t$ and $x$ and where $\bar b_t^\e$ is an approximating sequence for $\bar b_t$ satisfying \eqref{eq:r1}, \eqref{eq:r2}, \eqref{eq:r3} and \eqref{eq:r4}. Notice that if $b_\e$ is an approximating sequence for $b$, then $\bar{b}_t ^\e$ is an an approximating sequence for $\bar{b}_t$, where $\bar{b}_t ^\e$ is defined in the obvious way.
Given the sequence $\{Y_i^{\theta,\e}\}_{i=0}^{N}$ constructed in \eqref{eq:theta back}, we define the function
\begin{equation}\label{eq:theta back 2}
\begin{aligned}
\tilde{Y}^{N,\theta,\e}(t,s,x)&\coloneqq \sum_{i=0}^{N-1}\left(Y_{i}^{\theta,\e}+\frac{s-t_i}{h}(Y_{i+1}^{\theta,\e}-Y_i^{\theta,\e})\right)\chi_i^t(s),
\end{aligned}
\end{equation}
where with $\chi_i^t$ we denote the characteristic function of the set $[0,t]\cap[t_i,t_{i+1}]$,.
Finally, we define the sequence $u^h$ of approximating solutions of \eqref{eq:te} as follows
\begin{equation}\label{def:sol-teN}
u^h(t,\cdot )=u_0(\tilde{Y}^{N,\theta,\e}(t,t,\cdot )),
\end{equation}
where to simplify the notations we drop the superscript $\theta,\e$ on the function $u^h$. Note that $u^h$ does not solve \eqref{eq:te}. However, since we have rephrased the problem as a forward ODE, one might expect from the stability analysis developed in Sections 3 and 4 that $u^h$ is a good candidate to approximate the unique Lagrangian solution of~\eqref{eq:te} (see \eqref{def:sol-te}). Indeed the following corollary holds, whose proof we omit since it easily follows from Theorem \ref{t:main1}.
\begin{cor}\label{cor:main}
Let $b\in L^1((0,T);W^{1,p}(\R^d))\cap L^\infty((0,T)\times\R^d)$ be a divergence-free velocity field for some $1<p<\infty$ and denote by $Y$ the unique regular Lagrangian flow solving \eqref{eq:back}. For $\e\geq 0$, let $b_{\e}$ be an approximating sequence of $b$, and set $h(\e)=\e^{2\beta}$. Let $N\in\N$ and $\{Y_i^{\theta,\e}\}_{i=0}^N$ be any solution of \eqref{eq:theta back} and $\tilde{Y}^{N,\theta,\e}(t,s,x)$ defined as in \eqref{eq:theta back 2}. Then, the same conclusions of Lemma \ref{lem:jac_num_irr} and Theorem \ref{t:main1} hold replacing $\tilde X ^{N, \theta, \e} , X_i ^{\theta, \e}$ and $X$ with $\tilde Y ^{N, \theta, \e} , Y_i ^{\theta, \e}$ and $Y$ respectively.
\end{cor}
%{\color{red}N.B.: le stime si scrivono per $Y(t,s,x)$ dove la $t$ va tenuta fissa mentre sulla $s$ si ragiona come fatto precedentemente con la serie telescopica.}
We now prove Theorem \ref{t:main2}. We point out that, as in the previous sections, in the following theorem we consider a spatial approximation of the velocity field also in the case $p>d$. This leaves unchanged the rate of convergence of the approximating flow $\tilde Y^{h,\theta,\e}$ and allows us to use the compressibility Lemma \ref{lem:jac_num_irr}.

\begin{proof}[Proof of Theorem \ref{t:main2}]
We divide the proof in several steps.\\
\\
\underline{Step 1} Initial datum in $C^1 _c (\mathbb{R}^d)$.\\
\\
Assuming that $u_0\in C^1_c(\R^d)$ such that $\mathrm{supp}\, u_0\subset B_R$, we define the set $A_\eta$ 
\begin{equation}\label{def:Aeta}
A_\eta\coloneqq \{x\in\R^d: |\tilde{Y}^{N,\theta,\e}(t,t,x)-Y(t,t,x)|>\eta \}.
\end{equation}
We want to apply the estimate on the superlevel sets \eqref{stima sopralivelli}. Let us consider $\delta(\e):=\max\left\{\e^\alpha, \e^{2\beta}\right\}$ and $\eta=\sqrt{\delta(\e)}$. Then, by rewriting $\e$ as a function of $h$, for any $r>0$ we have that
\begin{equation}
\sup_{t\in[0,T]}\mathscr{L}^d(B_r\cap A_\eta)\leq \frac{C\left(d,p,r,T,\|b\|_{L^\infty},\|\nabla b\|_{L^p}\right)}{|\log h|} .
\end{equation}

Thus, we use the definition of $u^h$ and $u^L$ to write
\begin{align*}
\|u^h(t,\cdot)-u(t,\cdot)\|_{L^1}&=\int_{\R^d}|u^h(t,x)-u(t,x)| \de x\\
&=\int_{\R^d}|u_0(\tilde{Y}^{N,\theta,\e}(t,t,x))-u_0(Y(t,t,x))| \de x\\
&\leq \mathrm{Lip}(u_0)\int_{B_{R}}|\tilde{Y}^{N,\theta,\e}(t,t,x)-Y(t,t,x)|\de x\\
&\leq \mathrm{Lip}(u_0)\int_{B_{R}\cap A_\eta}|\tilde{Y}^{N,\theta,\e}(t,t,x)-Y(t,t,x)|\de x\\
&+ \mathrm{Lip}(u_0)\int_{B_{R}\setminus A_\eta}|\tilde{Y}^{N,\theta,\e}(t,t,x)-Y(t,t,x)|\de x\\
&\leq \mathrm{Lip}(u_0)\left(\frac{C\left(d,p,R,T,\|b\|_{L^\infty},\|\nabla b\|_{L^p}\right)}{|\log h|}+\sqrt{h}\right),
\end{align*}
where in the last line we used the definition of $\eta$.\\
\\
\underline{Step 2} Density argument in $L^1$.\\
\\
Let $\{u_0^\delta\}_{\delta>0}$ be any Lipschitz approximation of $u_0$ with compact support. 
Thus, we use the definition of $u^h$ and $u^L$ to write
\begin{align*}
\|u^h(t,\cdot)-u(t,\cdot)\|_{L^1}&=\int_{\R^d}|u^h(t,x)-u(t,x)| \de x\\
&=\int_{\R^d}|u_0(\tilde{Y}^{N,\theta,\e}(t,t,x))-u_0(Y(t,t,x))| \de x\\
&\leq \int_{\R^d}|u_0(\tilde{Y}^{N,\theta,\e}(t,t,x))-u_0^\delta(\tilde{Y}^{N,\theta,\e}(t,t,x))| \de x\\
&+\int_{\R^d}|u_0^\delta(\tilde{Y}^{N,\theta,\e}(t,t,x))-u_0^\delta(Y(t,t,x))| \de x\\
&+\int_{\R^d}|u_0(Y(t,t,x))-u_0^\delta(Y(t,t,x))| \de x.
\end{align*}
To provide an estimate on the first term we use Lemma \ref{lem:jac_num_irr}  to show that
\begin{equation}\label{dove non funziona}
\int_{\R^d}|u_0(\tilde{Y}^{N,\theta,\e}(t,t,x))-u_0^\delta(\tilde{Y}^{N,\theta,\e}(t,t,x))| \de x\leq C \|u_0^\delta-u_0\|_{L^1}.
\end{equation}
On the other hand, for the last term we use the definition of regular Lagrangian flow to compute
\begin{equation}
\int_{\R^d}|u_0^\delta(Y(t,t,x))-u_0(Y(t,t,x))| \de x\leq \|u_0^\delta-u_0\|_{L^1}.
\end{equation}
Finally, we use Step 1 to estimate the middle term
\begin{align*}
\int_{\R^d}\hspace{-0.2cm}|u_0^\delta(\tilde{Y}^{N,\theta,\e}(t,t,x))-u_0^\delta(Y(t,t,x))| \de x&=\int_{B_{R_\delta}}\hspace{-0.2cm}|u_0^\delta(\tilde{Y}^{N,\theta,\e}(t,t,x))-u_0^\delta(Y(t,t,x))| \de x\\
&\leq \mathrm{Lip}(u_0^\delta)\,\frac{C\left(d,p,R_\delta,T,\|b\|_{L^\infty},\|\nabla b\|_{L^p}\right)}{|\log h|}.
\end{align*}
Obviously, the constant above explodes as $\delta\to 0$ but $h$ and $\delta$ are not related to each other. Thus, we have obtained that
\begin{align*}
\|u^h(t,\cdot)-u(t,\cdot)\|_{L^1}&\lesssim \frac{C_\delta}{|\log h|}+ \|u_0^\delta-u_0\|_{L^1},
\end{align*}
where the implicit constant depends on $d,p,T,\|b\|_{L^\infty},\|\nabla b\|_{L^p}, \bar C_1, \bar C_2$. The convergence follows by choosing $\delta$ small enough and then $h$ small enough such that
$$
\|u_0^\delta-u_0\|_{L^1}\leq\gamma,\qquad \frac{C_\delta}{|\log h|}\leq \gamma,
$$
for any fixed $\gamma>0$.
\end{proof}

\begin{rem}
If $b\in L^1((0,T);W^{1,p}(\R^d))$ and $p<\infty$, for $u_0\in L^1(\R^d)$ distributional solutions of \eqref{eq:te} are not unique. Thus, Theorem \ref{t:main2} provides a selection principle in the sense that it rules out the convergence of the sequence $u^h$ towards a distributional solution different than the Lagrangian one.
\end{rem}

\appendix
\section{An example of a compatible space discretization} \label{sec:discr-space}
We now provide a nontrivial example of a numerical regularization in the space variable satisfying the hypotheses \eqref{eq:r1}-\eqref{eq:r4}. We remark that the classical convolution with a smooth kernel satisfies the aforementioned assumptions. 
The following construction is based on the {\em vortex-blob method} used in the theory of weak solutions of the 2D Euler equations, see \cite{Be, CCS3, DPM}. We recall that the vortex-blob method is the prototype of several numerical methods that are commonly used in the analysis of the Euler equations and it is based on the idea of approximating the vorticity with a finite number of cores which evolve according to the velocity of the fluid. We point out that the example provided in this section can be easily generalized to non-Cartesian meshes. Indeed, we decided to use the lattice $\Lambda_\ell$ below just to simplify the exposition.\\
\\
Let $\e\in (0,1)$ be a small parameter that will be chosen later. We consider the lattice
$$
\Lambda_\ell\coloneqq \{\alpha_i\in \ell\Z\times \ell\Z:\alpha_i=\ell(i_1,i_2), \mbox{ where }i_1,i_2\in\Z  \},
$$
and define $Q_i$ the square with sides of length $\ell$ parallel to the coordinate axis and centered at $\alpha_i\in \Lambda_\ell$. Let $\varphi_\e$ be a standard smooth mollifier. Assume for simplicity that $b\in W^{1,p}(\R^2;\R^2)$ for some $p>2$ is a divergence-free velocity field with compact support. Then, we define the vorticity of $b$:
$$
\omega\coloneqq \curl b=\partial_{x_2}b_1-\partial_{x_1}b_2,\qquad \mbox{and }\qquad
K(x)\coloneqq  \frac{x^\perp}{|x|^2}.
$$
From now on, since the parameter $\ell$ will be chosen dependent on $\e$, we will use only the superscript/subscript $\e$.
Thus, we define the quantities
\begin{equation}
    \Gamma_i^\e\coloneqq \int_{Q_i}\omega(x)\de x,
\end{equation}
and the function
\begin{equation}\label{def:app-num1}
\omega_\e(x)\coloneqq \sum_{i=1}^{N(\e)}\Gamma_i^\e \,\varphi_\e(x-\alpha_i).
\end{equation}
Before defining the approximating velocity field, we first prove the following lemma.
\begin{lem}\label{lem:vort}
Let $\omega\in L^1\cap L^p(\R^2)$ for some $p>2$ and let $\omega_\e$ be defined as in~\eqref{def:app-num1}. Let $\ell=\ell(\e):=\e^4$.
Then, the estimate
\begin{equation}\label{es:1}
\| \omega_\e-\varphi_\e*\omega\|_{L^p} \leq C \|\omega\|_{L^1}\e^{1+\frac{2}{p}}
\end{equation}
holds for all $1\leq p\leq \infty$, where $C>0$ is a positive constant which does not depend on~$\e$. Moreover, the following bound holds
\begin{equation}
\|\omega_\e\|_{L^p}\leq C\|\omega\|_{L^1}+\|\omega\|_{L^p},
\end{equation}
for some positive constant $C>0$.
\end{lem}
\begin{proof}
We start by proving the inequality \eqref{es:1} in the case $p=1$. By using the definition of $\omega_\ell$ we have that
\begin{align}
\int_{\R^2} &\left|\omega_\e(x)-\varphi_\e*\omega(x)\right|\de x = \int_{\R^2}\left|
\displaystyle\sum_i \Gamma^\e_i \varphi_\e(x-\alpha_i)-\int_{\R^2} \varphi_\e(x-z)\omega(z) dz
\right| \de x \nonumber \\
&=\int_{\R^2} \left|
\displaystyle\sum_i \left(\int_{Q_i}\omega(y)\de y \right)\,\varphi_\e(x-\alpha_i)-\sum_i \int_{Q_i} \varphi_\e(x-y)\omega(y) \de y \right|\de x \nonumber \\
&=\int_{\R^2} \left|
\displaystyle\sum_i \int_{Q_i} \omega(y)\left[\varphi_\e(x-\alpha_i)-\varphi_\e(x-y)\right] \de y \right| \de x \nonumber \\
&\leq \int_{\R^2} \sum_i \int_{Q_i} \left| \omega(y)\right| |\underbrace{\varphi_\e(x-\alpha_i)-\varphi_\e(x-y)}_{(*)}| \de y \de x \label{resc}.
\end{align}
For the term $(\ast)$ we have the following estimate
\begin{equation*}
\varphi_\e(x-\alpha_i)-\varphi_\e(x-y)= \int_0^1 \nabla\varphi_\e(x-s\alpha_i-(1-s)y) \de s\cdot \left(y-\alpha_i\right),
\end{equation*}
and since $y,\alpha_i\in Q_i$, we have that $|y-\alpha_i|\leq \sqrt{2}\ell$.
Then, integrating in the $x$ variable in \eqref{resc} we get
$$
\|\omega_\e-\varphi_\e*\omega\|_{L^1} \leq \frac{\ell}{\e}\, \|\nabla\varphi\|_{L^1}\|\omega\|_{L^1}.
$$
Choosing the function $\ell(\e)=\e^4$, we get \eqref{es:1} for $p=1$.\\
For $p=\infty$ we can argue in a similar way; by the same computations as in \eqref{resc} we have that, since $||\nabla \phi_\e||_\infty \leq \e  ^{-3} ||\nabla \phi||_\infty$:
\begin{equation}
|\omega_\e(x)-\varphi_\e*\omega(x)|\leq \frac{\ell}{\e^3}\|\nabla\varphi\|_\infty\|\omega\|_{L^1}.
\end{equation}
Finally, by interpolating we have
$$
\|\omega_\e-\varphi_\e*\omega\|_{L^p}\leq \|\omega_\e-\varphi_\e*\omega\|_{L^1}^{\frac{1}{p}}\|\omega_\e-\varphi_\e*\omega\|_{L^\infty }^{1-\frac{1}{p}}\leq C\|\omega\|_{L^1}\e^{1+\frac{2}{p}},
$$
and this concludes the proof of \eqref{es:1}. The $L^p$ bound on $\omega_\e$ follows from 
$$
\|\omega_\e\|_{L^p}\leq \|\omega_\e-\varphi_\e*\omega\|_{L^p}+\|\varphi_\e*\omega\|_{L^p}\leq C \|\omega\|_{L^1}+\|\omega\|_{L^p},
$$
and this conclude the proof.
\end{proof}
We can now define the approximating velocity field. Denoting $K_\e\coloneqq K*\varphi_\e$ and assuming $\ell(\e)=\e^4$, we consider:

\begin{equation}
\label{def:appr-num2}
b_\e(x)=\sum_{i=1}^{N(\e)}\Gamma_i^\e \,K_\e(x-\alpha_i)= \sum_{i=1}^{N(\e)}\Gamma_i^\e \,\int_{\R^2} K(x)  \varphi_\e (x- \alpha_i) \de x = K * \omega_\e (x) .
\end{equation}

Notice that the sum in \eqref{def:appr-num2} is finite because of the assumption on the support of $b$. We now verify that \eqref{def:appr-num2} satisfies the assumptions \eqref{eq:r1}-\eqref{eq:r4}.
\begin{lem}
The velocity field $b_\e$ defined in \eqref{def:appr-num2} is smooth and divergence-free. Moreover, if $\ell(\e)=\e^4$, there exists a constant $C>0$ such that the sequence $b_\e$ satisfies the bounds \eqref{eq:r1}-\eqref{eq:r4} with $\alpha=1$ and $\beta=3$.
\end{lem}
\begin{proof}
First of all, the divergence-free condition and the smoothness of $b_\e$ follows by its definition. Thus, condition \eqref{eq:r3} is satisfied. Moreover, we recall that we can split $K\coloneqq K_1+K_2$ where $K_1\in L^q(\R^2)$ for any $1\leq q< 2$, while $K_2\in L^q(\R^2)$ for any $2<q\leq \infty$. We start by proving \eqref{eq:r1}: we have that
\begin{align*}
    \|b_\e-b\|_{L^p}&\leq  \|b_\e-\varphi_\e*b\|_{L^p}+ \|\varphi_\e*b-b\|_{L^1}\\
    &\leq \|K*\omega_\e-\varphi_\e*K*\omega\|_{L^p}+\e\|\nabla b\|_{L^p}\\
    &= \|K*(\omega_\e-\varphi_\e*\omega)\|_{L^p}+\e\|\nabla b\|_{L^p}\\
    &\leq \|K_1*(\omega_\e-\varphi_\e*\omega)\|_{L^p}+\|K_2*(\omega_\e-\varphi_\e*\omega)\|_{L^p}+\e\|\nabla b\|_{L^p}\\
    &\leq \|K_1\|_{L^1}\|\omega_\e-\varphi_\e*\omega\|_{L^p}+\|K_2\|_{L^p}\|\omega_\e-\varphi_\e*\omega\|_{L^1}+\e\|\nabla b\|_{L^p}\\
    &\leq C (\e^{1+\frac2p}+\e^{3}+\e) \lesssim \e,
\end{align*}
where we used Lemma \ref{lem:vort} in the last line, and we passed from a $L^1$ estimate in the first line to a $L^p$ estimate in the second one by the compact support of $b$. Thus, the conclusion follows.
We now prove \eqref{eq:r2}: we use  
\begin{align*}
\|\nabla b_\e\|_{L^\infty}&\leq\sum_{i=1}^{N(\e)}|\Gamma_i^\e|( \|K_1*(\nabla\varphi_\e)\|_{L^\infty}+\|K_2*(\nabla\varphi_\e)\|_{L^\infty})\\
&\leq ||\omega||_{L^1} ( \|K_1*(\nabla\varphi_\e)\|_{L^\infty}+\|K_2*(\nabla\varphi_\e)\|_{L^\infty}) \\&\leq ||\omega||_{L^1} ( ||K_1||_{L^1} \; ||\nabla\varphi_\e ||_{L^\infty} +||K_2||_{L^\infty} ||\nabla\varphi_\e||_{L^1}) \leq C\frac{\|\omega\|_{L^1}}{\e^3},
\end{align*}
where the constant $C$ depends on $\|K_1\|_{L^1},\|K_2\|_\infty,\|\nabla \varphi\|_{L^1},$ and $\|\nabla\varphi\|_\infty$. Lastly, we prove \eqref{eq:r4}: we split $K$ as above and we observe that
\begin{align*}
    \|b_\e\|_\infty&\leq \sum_{i=1}^{N(\e)}|\Gamma_i^\e|\|K*\omega_\e\|_\infty\leq \|\omega\|_{L^1}\left(\|K_1*\omega_\e\|_\infty+\|K_2*\omega_\e\|_\infty\right)\\
    &\leq \|\omega\|_{L^1}\left(\|K_1\|_{L^q}\|\omega_\e\|_{L^p}+\|K_2\|_{L^\infty}\|\omega_\e\|_{L^1}\right)\\
    &\leq C\|\omega\|_{L^1}(\|\omega\|_{L^1}+\|\omega\|_{L^p}),
\end{align*}
where we have used that $p>2$ when we apply Young's inequality in the estimate for the term $K_1*\omega_\e$ in order to choose $q <2$, together with \eqref{es:1}. The proof is complete.
\end{proof}

\section*{Acknowledgments}
G. Crippa is supported by SNF Project 212573 FLUTURA – Fluids, Turbulence, Advection and by the SPP 2410 ``Hyperbolic Balance Laws in Fluid Mechanics: Complexity, Scales, Randomness (CoScaRa)'' funded by the Deutsche Forschungsgemeinschaft (DFG, German Research Foundation) through the project 200021E\_217527 funded by the Swiss National Science Foundation. 
G. Ciampa is supported by European Union - NextGenerationEU - National Recovery and Resilience Plan (Piano Nazionale di Ripresa e Resilienza, PNRR) - National Centre for HPC, Big Data and Quantum Computing (CN\_00000013 - CUP: E13C22001000006). 
The work of G. Ciampa and S. Spirito is partially supported by INdAM-GNAMPA through the project ``Stabilit\`a e confronto per equazioni di trasporto/diffusione e applicazioni a EDP non lineari'', and by the projects PRIN 2020 ``Nonlinear evolution PDEs, fluid dynamics and transport equations: theoretical foundations and applications”, and PRIN2022 ``Classical equations of compressible fluid mechanics: existence and properties of non-classical solutions''. T. Cortopassi is a member of the INdAM group GNAMPA. R. D'Ambrosio is a member of the INdAM group GNCS. The work of R. D'Ambrosio is partially supported by GNCS-INdAM project and by PRIN-PNRR project P20228C2PP ``BAT-MEN (BATtery Modeling, Experiments 
\& Numerics) - Enhancing battery lifetime: mathematical modeling, numerical simulations and AI parameter estimation techniques for description and control of material localization processes''. The authors thank Michele Benzi and Jan Giesselmann for valuable discussions and feedback on an earlier version of this paper.

\end{document}